\documentclass{article}
\usepackage[utf8]{inputenc}

\usepackage{amsmath}
\usepackage{mathrsfs}
\usepackage{amsthm}
\usepackage{multirow}
\usepackage{tikz-cd}
\usepackage{rotating}
\usepackage{listings}
\usepackage{amssymb}
\usepackage{booktabs}
\usepackage{caption}
\usepackage{algorithm2e}
\usepackage{float}
\usepackage{placeins}
\usepackage{graphicx}
\usepackage{subcaption}
\usepackage[a4paper,left=1.905cm, right=1.905cm, top=1.7cm]{geometry}
\usepackage{adjustbox}
\usepackage{xparse}
\usepackage{amsmath}
\usepackage{environ}
\usepackage{blkarray}
\usepackage[hidelinks]{hyperref}
\hypersetup{
allcolors=black
    colorlinks=true, 
    linktoc=all,     
    linkcolor=black,  
}
\RequirePackage{doi}
\usepackage{doi}
\ExplSyntaxOn
\NewEnviron{pmatrixT}
 {
  \marine_transpose:V \BODY
 }

\int_new:N \l_marine_transpose_row_int
\int_new:N \l_marine_transpose_col_int
\seq_new:N \l_marine_transpose_rows_seq
\seq_new:N \l_marine_transpose_arow_seq
\prop_new:N \l_marine_transpose_matrix_prop
\tl_new:N \l_marine_transpose_last_tl
\tl_new:N \l_marine_transpose_body_tl

\cs_new_protected:Nn \marine_transpose:n
 {
  \seq_set_split:Nnn \l_marine_transpose_rows_seq { \\ } { #1 }
  \int_zero:N \l_marine_transpose_row_int
  \prop_clear:N \l_marine_transpose_matrix_prop
  \seq_map_inline:Nn \l_marine_transpose_rows_seq
   {
    \int_incr:N \l_marine_transpose_row_int
    \int_zero:N \l_marine_transpose_col_int
    \seq_set_split:Nnn \l_marine_transpose_arow_seq { & } { ##1 }
    \seq_map_inline:Nn \l_marine_transpose_arow_seq
     {
      \int_incr:N \l_marine_transpose_col_int
      \prop_put:Nxn \l_marine_transpose_matrix_prop
       {
        \int_to_arabic:n { \l_marine_transpose_row_int }
        ,
        \int_to_arabic:n { \l_marine_transpose_col_int }
       }
       { ####1 }
     }
   }
   \tl_clear:N \l_marine_transpose_body_tl
   \int_step_inline:nnnn { 1 } { 1 } { \l_marine_transpose_col_int }
    {
     \int_step_inline:nnnn { 1 } { 1 } { \l_marine_transpose_row_int }
      {
       \tl_put_right:Nx \l_marine_transpose_body_tl
        {
         \prop_item:Nn \l_marine_transpose_matrix_prop { ####1,##1 }
         \int_compare:nF { ####1 = \l_marine_transpose_row_int } { & }
        }
      }
     \tl_put_right:Nn \l_marine_transpose_body_tl { \\ }
    }
   \begin{pmatrix}
   \l_marine_transpose_body_tl
   \end{pmatrix}
 }
\cs_generate_variant:Nn \marine_transpose:n { V }
\cs_generate_variant:Nn \prop_put:Nnn { Nx }
\ExplSyntaxOff

\theoremstyle{definition}

\newtheorem{theorem}{Theorem}[section] 

\newtheorem{corollary}[theorem]{Corollary} 

\newtheorem{lemma}[theorem]{Lemma} 

\newtheorem{fact}[theorem]{Fact} 

\newtheorem{definition}[theorem]{Definition} 

\newtheorem{example}[theorem]{Example}
 
\makeatletter
\makeatletter

\newcommand{\subsetsim}{\mathrel{\mathpalette\subset@sim\relax}}

\newcommand{\subset@sim}[1]{%
  \vtop{\offinterlineskip\m@th
    \ialign{\hfil##\cr
      $#1\subset$\cr\noalign{\kern0.3pt}\scalebox{0.9}{$#1\sim$}\cr
    }%
  }%
}
\makeatother

\newcommand\unitcolsum{\mathcal{M}}
\newcommand\zerocolsum{\mathcal{L}}
\newcommand\fra[2]{\textstyle{\frac{#1}{#2}}}
\newcommand\RR{\mathbb{R}}

\newcommand\oneVec{\pmb{\theta}}

\DeclareMathOperator{\spam}{span}
\DeclareMathOperator{\mat}{Mat}

\DeclareMathOperator{\fix}{Fix}
\newcommand\nnfootnote[1]{%
  \begin{NoHyper}
  \renewcommand\thefootnote{}\footnote{#1}%
  \addtocounter{footnote}{-1}%
  \end{NoHyper}
}

\title{Uniformization stable Markov models and their Jordan algebraic structure}
\author{Luke Cooper\thanks{Corresponding author},\ Jeremy Sumner\\\\{\small School of Natural Sciences, University of Tasmania,
Private Bag 37,
Hobart, TAS, 7001, Australia}}

\begin{document}
\maketitle
\begin{abstract}
\noindent We provide a characterisation of the continuous-time Markov models where the Markov matrices from the model can be parameterised directly in terms of the associated rate matrices (generators). That is, each Markov matrix can be expressed as the sum of the identity matrix and a rate matrix from the model.
We show that the existence of an underlying Jordan algebra provides a sufficient condition, which becomes necessary for (so-called) linear models. We connect this property to the well-known uniformization procedure for continuous-time Markov chains by demonstrating that the property is equivalent to all Markov matrices from the model taking the same form as the corresponding discrete time Markov matrices in the uniformized process.
We apply our results to analyse two model hierarchies practically important to phylogenetic inference, obtained by assuming (i) time-reversibility and (ii) permutation symmetry, respectively.
\\
Keywords: Markov chain, Jordan algebra, representation theory, phylogenetics, time reversibility \\
MSC: 60J10, 17C90

\end{abstract}

\section{Introduction}
\label{sec:introduction}
Although the results we present here are general and valid for a wide range of modelling purposes, our motivations primarily come from the mathematics that underlies modern approaches to phylogenetic inference.
The latter can broadly be understood as the suite of mathematical, statistical and computational methods currently available to infer evolutionary trees from present-day sequence data, which are excellently summarised in \cite{felsenstein,steel}.
In particular, these methods typically involve continuous-time Markov chains on a finite state space (e.g. four states for models of DNA evolution) and specific models occur as collections of Markov generator matrices, or `rate matrices', where the model is defined by polynomial restrictions on the matrix entries (at least implicitly --- we will present several examples below). 
From the statistical inference point of view, we note that the parameterisation of a phylogenetic model most often occurs at the level of the substitution rates, as opposed to substitution probabilities, and it is this formulation that motivates the ideas we explore here. \nnfootnote{Email addresses: \url{luke.cooper@utas.edu.au} (L. Cooper) and \url{jeremy.sumner@utas.edu.au} (J. Sumner)}

Historically, the application of abstract algebra to genetics has been considered from as early as 1939 \cite{genAlgebra}.
In particular, Jordan algebras have been discussed in the context of genetics for the theory of DNA recombination \cite{dnaRecomb} and, in a broader context, for Markov processes on non-associative state spaces \cite{mukhamedov}. 
The work we present here provides a detailed examination of the natural Jordan algebraic structure that underlies many of the DNA substitution models commonly used in phylogenetic inference. 
We explain the details of how this structure arises presently.

We recall the well-known result from Markov chain theory that, for the homogeneous continuous time case, the associated Markov (stochastic) matrices and rate matrices (generators) are linked by the matrix exponential $M=e^{Qt}$, where $t$ is time elapsed, $M$ is a Markov matrix (non-negative entries and row sums equal to 1), and $Q$ is a rate matrix (non-negative off-diagonal entries and row sums equal to 0). 
We also recall the classical `Markov embedding problem', which asks, given a Markov matrix $M$, whether there exists a rate matrix $Q$ such that $M\in \{e^{Qt}:t\geq 0\}$ \cite{davies}.
In recent work on addressing the embedding problem, \cite{sumnerEmbed} introduced the important notion of `stability' for Markov models and, in particular, noted that, in the specific case where the polynomial restrictions defining a model are linear, the stability condition is equivalent to demanding the rate matrices of the model form a Jordan algebra (in the appropriate sense, with details given below). 
In the present work we establish that the class of models where the rate matrices form a Jordan algebra are precisely those for which there is a elementary relationship between the structure of the rate matrices $Q$ and the resulting Markov matrices $M=e^{Qt}$. 
Specifically, we show that requiring the set of rate matrices defining a Markov model to form a Jordan algebra is equivalent to the statement that, for each rate matrix $Q$ in the model and time $t>0$, there exists another rate matrix $\hat Q$ in the same model such that $e^{Qt}-I_n=\hat Q$ (where $I_n$ is the $n\times n$ identity matrix). 
We will formally define this notion using the standard notion of uniformization for continuous-time Markov chains, and hence refer to models satisfying this property as `uniformization stable'.

From a modelling perspective, this structural connection is compelling as it guarantees the most elementary relationship between the structure of the rate matrices and the Markov matrices of the model.
For instance, this relationship provides a direct method for parameterising a model in terms of substitution probabilities, rather than using rates (as is standard practice in the practical setting). More specifically, uniformization stability of a model implies that the relationship between the Markov matrices and the rate matrices is linear.
For illustrative purposes, we present examples of well-known models taken from molecular phylogenetics where this structural connection is present as well as some where it is absent. Additionally, uniformization stability may be of more general interest due to the potential to lead to improvements in the efficiency and numerical accuracy of matrix exponential computation.

Besides our theoretical discussions, the main goal of this article is to construct application-ready hierarchies of uniformizaton stable Markov models.
As we will argue, naively attempting to list all Markov models that form Jordan algebras is theoretically misguided and computationally infeasible; the problem being that there is actually a continuous infinity of such models available, whereas only a finite subset of these are plausibly of practical interest.
To this end, we adopt the method presented in \cite{sumnerLie} for the related context of `Lie-Markov' models and show that this allows us to systematically identify the finite set of interesting Markov models that form Jordan algebras.
The method given in that prior work relies on discrete model symmetries (specifically, equivalences of nucleotides under permutations) and applies the algebraic theory of group representations to systematically produce a finite hierarchy of Lie-Markov models, given a specific symmetry group of sufficient size.
In particular, the case where the nucleotides are pairwise equivalent leads to analysis with the symmetric group on four elements, $S_4$. 
In Section~\ref{sec:symmetry} below, we explore the situation for Jordan algebras, giving a complete characterisation for the general case of $n$ states.

As is developed below, the approach given in \cite{sumnerLie} was originally conceived for the construction of hierarchies of Lie-Markov models but is easily adaptable to our context of Jordan algebras.
We hence follow \cite{sumnerLie} in this work, with the following notable technicality and generalization.  

In practice, many Markov models on $n$ states are implicitly (if not explicitly) chosen to adhere to a finite group $G\leq S_n$ of state permutations.
As discussed above, imposing such symmetries reduces the number of candidate models from an unworkable continuous infinity to a finite set (as long as the symmetry group $G$ is sufficiently large).  
As we will describe in detail, this leads to a classification of uniformization stable Markov models with (so-called) $G$-symmetry. 
From a biological perspective, in the case of Markov models of DNA evolution (models on 4 states), demanding these symmetries is completely natural: for example, the group $D_4<S_4$ of dihedral permutations arises naturally from the partitioning of nucleotides into purine and pyrimidine substitutions, with the classification of Lie-Markov models with this symmetry presented in \cite{sumnerPurine}. 
Further, we relax an unnecessary restriction taken in \cite{sumnerLie} by removing the requirement that models with $G$-symmetry must also have a `permutation basis'.

In Section \ref{sec:structure} we rigorously define the previously discussed concepts and give all the necessary convex and algebraic structures relating to Markov rate matrices and useful basic properties of Jordan algebras of matrices. 
We also give the main result of this article (Theorem \ref{thm:jordanIffExponential}), showing that a linear Markov model is uniformization stable if and only if an underlying Jordan algebra is present. 
In Section \ref{sec:timeReversible} we apply this result to derive a hierarchy of (non-linear) time-reversible uniformization stable Markov models, and, in Section \ref{sec:symmetry}, we present for each $n\geq 2$ a precise characterisation of the hierarchy of $S_n$-symmetric linear uniformization stable models. 

\section{Algebraic structures of Markov matrices}
\label{sec:structure}
We begin with a detailed discussion of the two classes of matrices that we will be working with throughout. 
Although some authors define a `Markov model' as a collection of Markov probability matrices, consistent with the introductory remarks above, we focus on the continuous-time formulation and will define a Markov model in terms of rate matrices (or generators). 
In particular, Lemma \ref{lem:rateExponentialSubset} provides us with a natural way to associate a set of Markov matrices with a Markov model defined using rate matrices. 
This allows us to rigorously define a Markov model for the purposes of the rest of this article and will motivate the main topic of interest: `uniformization stable' Markov models (Definition~\ref{def:expclosed}).

As the matrices used in Markov chain theory represent stochastic concepts such as probabilities and substitution rates, we ultimately work in convex spaces and implement the appropriate non-negativity constraints to the relevant matrices.
However, when we consider these matrices as sets equipped with an algebraic structure, we will relax the non-negativity conditions as appropriate to our purposes. 
As an example to foreshadow what is to come: any (real) linear space of matrices containing a rate matrix $Q\neq 0$ with non-negative off-diagonal entries must also contain the matrix $-Q$, which is no longer a rate matrix.

For clarity, we presently define the concept of unit row sum and zero row sum matrices, then define Markov probability and rate matrices as subsets of these below. 
We also begin to consider the algebraic structure of these more general sets.

\begin{definition}
\label{def:unitRowSum}
 A \emph{unit row sum matrix} $M$ is an $n \times n$ matrix such that the entries in each row sum to $1$. 
 For fixed $n$, we denote the collection of these matrices as $\unitcolsum_{n}$. 
 That is, taking $\oneVec$ to denote the column vector of $n$ $1$'s:
\[
\unitcolsum_{n}= \{M\in \mat_n(\RR) \ : \ M \oneVec =\oneVec\} 
\]

\end{definition}

Since $\oneVec=M_1 \oneVec = M_1(M_2 \oneVec )=(M_1M_2) \oneVec $ holds for all $M_1,M_2\in \unitcolsum_n$, and the identity matrix $I_n\in \unitcolsum_n$, we have:

\begin{fact}
The set of real unit row sum matrices $\mat_n(\RR)$ forms a monoid --- that is, a semigroup with identity --- under standard matrix multiplication.

\end{fact}

Recall that a \emph{(real) affine combination} in a linear space $S$ is a linear combination $\lambda_1s_1 + \lambda_2 s_2 + \ldots + \lambda_n s_n$, where $s_1,s_2,\ldots,s_n \in S$, $\lambda_1,\lambda_2,\ldots,\lambda_n \in \RR$ and $\lambda_1+\lambda_2+\ldots+\lambda_n = 1$. 
We have,
\begin{fact}
The set of unit row sum matrices is closed under real affine combinations. That is, $\unitcolsum_n$ is an affine subset of $\mat_n(\RR)$.
\end{fact}

\begin{definition}
\label{def:zeroRowSum}
A \emph{zero row sum matrix} $Q$ is an $n \times n$ matrix such that the entries of each row of $Q$ sum to $0$. 
We denote the collection of all such matrices for fixed $n$ as follows:
$$\zerocolsum_{n}= \{Q \in \mat_n(\RR) \ : \ Q \oneVec=0 \}. $$

\end{definition}

In particular noting that $I_n\notin \zerocolsum_n$, we have:

\begin{fact}
\label{fact:rateSubAlgebra}
The zero row sum matrices $\zerocolsum_{n}$ form a matrix subalgebra of $\mat_n(\RR)$.
The right-annihilator of $\zerocolsum_n$ is non-trivial; hence $\zerocolsum_n$ is not a semisimple algebra.
Finally, $\zerocolsum_{n}$ contains a family of right identities, but no left identity.

Specifically, for all $Q, Q_1, Q_2 \in \zerocolsum_{n}$, $\lambda \in \RR$ we have, 
\begin{enumerate}
    \item  $Q_1 + \lambda Q_2 \in \zerocolsum_{n}$.
    \item $Q_1Q_2 \in \zerocolsum_{n}$.
    \item There exists (multiple) non-zero $B\in \zerocolsum_n$ such that $QB=0$.
    \item There exists (multiple) $ R\in \zerocolsum_n$ such that $QR=Q$.
    \item There is no non-zero $C\in \zerocolsum_n$ such that $CQ=0$.
    \item There is no $L\in \zerocolsum_n$ such that $LQ=Q$.
\end{enumerate}

\end{fact}
\begin{proof}
Consider $Q_1, Q_2 \in \zerocolsum_{n}$ and $\lambda \in \RR $. Then,
\[
(Q_1+\lambda Q_2)\oneVec  =Q_1\oneVec  +  (\lambda Q_2)\oneVec=0+\lambda ( Q_2\oneVec )=\lambda \cdot 0=0
= Q_2 \oneVec
=Q_1 (Q_2\oneVec)=
(Q_1Q_2)\oneVec ,
\]
hence $(1.)$ and $(2.)$ are valid.

Consider the family of matrices 
\[
B=
\begin{pmatrixT}
\alpha_1 & \alpha_1 & \vdots & \alpha_1\\
\alpha_2 & \alpha_2 & \vdots & \alpha_2\\
\ldots & \ldots &\ddots & \ldots \\
\alpha_n & \alpha_n & \vdots & \alpha_n
\end{pmatrixT}
\in \zerocolsum_n,
\]
with $\sum_{i=1}^n\alpha_i=0$.
The row sum condition implies $QB=0$ for all $Q\in \zerocolsum_n$; and hence $(3.)$ holds, which, in particular, establishes that $\zerocolsum_n$ is not semisimple. 

Consider the $n\times n$ Markov matrix $H= \textstyle{\frac{1}{n}}\oneVec \oneVec^T \in \unitcolsum_n$, which has every entry equal to $\textstyle{\frac{1}{n}}$, and observe that $J:=H-I_n\in \zerocolsum_n$ is a rate matrix.
Taking $Q\in \zerocolsum_n$ and any $B$ as above, we compute
\[
Q(-J+B)=-QH+Q+QB=-0+Q+0=Q,
\]
and hence set $R:=-J+B$ and apply $(1.)$ to see that $(4.)$ holds.

Now suppose $C\in \zerocolsum_n$ satisfies $CQ=0$ for all $Q\in \zerocolsum_n$.
Using $(4.)$, we immediately obtain $C=CR=0$ and hence conclude that $(5.)$ is valid.

Finally, suppose $L\in\zerocolsum_n$ is a left identity in $\zerocolsum_n$. 
Again taking any $B\neq 0$ as above, we find:
\[
B=LB=0,
\]
which is a contradiction and we conclude that $(6.)$ holds.

\end{proof}

The interested reader can readily verify that \emph{all} right identities of $\zerocolsum_n$ actually occur in the form $-J+B$, as presented in the proof of Fact~\ref{fact:rateSubAlgebra}.
We have also drawn specific attention to the matrix $J=H-I_n=\textstyle{\frac{1}{n}}\oneVec \oneVec^T -I_n\in \zerocolsum_n$, as this will play a fundamental role in Section \ref{sec:symmetry}, particularly to the proof of Lemma \ref{lem:trivialMarkovModel}.
For the moment, one may observe that $-J$ is the unique right identity of $\zerocolsum_n$ that is invariant to simultaneous row and column permutations, and, in particular, $J^2=-J$.

 By adding appropriate non-negativity conditions, the $n\times n$ Markov matrices and $n\times n$ rate matrices (or generators) are naturally defined as subsets of the sets discussed thus far:

\begin{definition}
\label{def:generalMarkov}
 A \emph{Markov matrix} $M$ is a non-negative $n \times n$ unit row sum matrix. 
 We denote the collection of these matrices as $\unitcolsum_{n}^+$. 
 Of course $\unitcolsum_{n}^+ \subset \unitcolsum_{n}$ for all $n>1$.

 \end{definition}
 
Noting that the non-negativity implies the entries of a Markov matrix must actually lie in $[0,1]$, we have:
\[
\unitcolsum_{n}^+= \{M = (m_{ij})\in \mat_n(\RR) \ : \ M \oneVec =\oneVec , \ m_{ij} \in [0,1] \ \forall \ i,j=1,2,...,n \},
\]
and since the identity matrix $I_n\in \unitcolsum_n^+$ and non-negativity and unit row sum conditions are preserved under matrix multiplication, we have:

\begin{fact}
\label{fact:markovMonoid}
The set of Markov matrices forms a monoid under standard matrix multiplication.
\end{fact}

Further, recalling that a \emph{convex combination} is an affine combination where the scalars are restricted to be non-negative, we find:

\begin{fact}
\label{fact:markovConvex}
The set of Markov matrices is closed under convex combinations. 
That is $\unitcolsum_{n}^+$ is a convex subset of $\mat_{n}(\RR)$.
\end{fact}
As the entries of Markov matrices are constrained by the closed interval $[0,1]$, we also have that:
\begin{fact}
\label{fact:markovLimits}
The set of Markov matrices is a topologically closed subset of $\text{Mat}_n(\RR)$ (i.e. it is closed under limits).
\end{fact}

In the zero row sum case, we have:
\begin{definition}
\label{def:markovRateMatrix}
A \emph{rate matrix} $Q$ is an $n \times n$ zero row sum matrix such that each non-diagonal entry is non-negative. 
For fixed $n$, we denote the set of these matrices as $\zerocolsum_{n}^+$. 

\end{definition}
We may write, 
\[
\zerocolsum_{n}^+= \{Q=(q_{ij}) \in \mat_n(\RR) \ : \ Q \oneVec =0, \ q_{ij} \geq 0 \ \forall \ i,j=1,2,...,n, \ i \neq j \}\subset \zerocolsum_{n},
\]
and, recalling that a \emph{conical combination} in a real linear space is a linear combination where the coefficients are restricted to be non-negative, we have:
\begin{fact}
\label{fact:ratematricesconical}
The set of rate matrices $\zerocolsum_n^+$ is closed under conical combinations.
Thus, $\zerocolsum_n^+$ is a \emph{convex cone}.
\end{fact}

The following results provide the explicit connection between rate matrices and Markov matrices in the context of a continuous-time Markov chain. 
We also explicitly define what we mean by a Markov model.

\begin{lemma}
\label{lem:rateExponentialSubset}
Given a subset $\zerocolsum \subseteq \zerocolsum_{n}$, the set 

$$e^{\zerocolsum} := \{e^Q \ : \ Q \in \zerocolsum \} $$
is a subset of $\unitcolsum_{n}$.
\end{lemma}

\begin{proof}
Consider an arbitrary element $M \in e^{\zerocolsum}$. Then, $M=e^Q$ for some $Q \in \zerocolsum_{n}$. By the definition of the matrix exponential,

$$M=e^Q=\sum_{i=0}^{\infty}\frac{Q^i}{i!}=I_n+\left(Q+\frac{Q^2}{2!}+\frac{Q^3}{3!}+\ldots\right).$$

Now by Fact \ref{fact:rateSubAlgebra}, each $\frac{Q^k}{k!}\in \zerocolsum_n$ and, 
as the series representation of the exponential map has infinite radius of convergence, we have $e^Q=I_n+\hat{Q}$ for some $\hat{Q} \in \zerocolsum_{n}$. Observe that $e^Q\oneVec=I_n\oneVec+\hat{Q}\oneVec=\oneVec$, and hence $M=e^Q \in \unitcolsum_{n}$.

\end{proof}

Of course, the more specific result that $e^Q\in \unitcolsum_n^+$ whenever $Q\in \zerocolsum_n^+$ is also valid.
We establish this presently but in a way that preempts our generalization to model-specific contexts (Definition~\ref{def:expclosed}, below).

\begin{definition}
\label{def:markovModel}

A \emph{Markov model} $\zerocolsum^+$ is any subset of $\zerocolsum_n^+$.
We say a Markov model is \emph{linear} if $\zerocolsum^+:=\zerocolsum \cap \zerocolsum_{n}^+$, where $\zerocolsum$ is a linear subspace of $\zerocolsum_{n}$.

\end{definition}

In practice many of the Markov models that are used in phylogenetics, for example, can be viewed as $\zerocolsum^+:=\zerocolsum \cap \zerocolsum_{n}^+$, where $\zerocolsum\subseteq \zerocolsum_n$ is defined using a fixed set of polynomial constraints on the entries of each matrix in $\zerocolsum$ (in other words, $\zerocolsum$ is an \emph{algebraic variety} in $\mat(n,\RR)\cong \RR^{n^2}$).
Under Definition~\ref{def:markovModel}, the linear Markov models are a particular subclass of these. 
We give an example of a non-linear Markov model in Example \ref{exam:nonLinearModel}, and linear Markov models are demonstrated in Examples \ref{exam:nonMinimalModel} and \ref{exam:jordanUniformization}.

We pause here to confirm that the exponential map sends rate matrices to Markov matrices:

\begin{lemma}
\label{lem:positiveRateExponentialSubset}
For any Markov model $\zerocolsum^+\subseteq \zerocolsum_n^+$, it follows that:

$$e^{\zerocolsum^+} = \left\{e^Q \ : \ Q \in \zerocolsum^+ \right\}\subseteq \unitcolsum_{n}^+.$$
\end{lemma}
\begin{proof}
Consider any $Q \in \zerocolsum^+$. 
Recall that the exponential map may be expressed as
$e^{Q}=\lim_{k \rightarrow \infty} (I+\frac{1}{k}Q)^k$ \cite[Theorem 10.1]{higham}.
Now, for sufficiently large $k$, the matrix $I+\frac{1}{k}Q$ is an element of $\unitcolsum_{n}^+$ and hence, as $\unitcolsum_{n}^+$ forms a monoid and is closed under limits (Facts \ref{fact:markovMonoid} and \ref{fact:markovLimits}), it follows that $e^{Q} \in \unitcolsum_n^+$. 
\end{proof}

It is worth noting here that, in general, $e^{\zerocolsum^+}$ is not closed under multiplication and hence is not a submonoid of $\unitcolsum_n^+$.
In particular the embedding problem produces several such notable examples (see \cite{sumnerEmbed}).

While we may consider any $\zerocolsum^+\subseteq \zerocolsum_n^+$ to be a Markov model, we usually assume $\zerocolsum^+=\zerocolsum\cap \zerocolsum_n^+$ and examine the algebraic structures present in $\zerocolsum\subseteq \zerocolsum_n$. 
In particular, the main classification result we present (Theorem~\ref{thm:jordanIffExponential} below) is formulated precisely for linear Markov models, in the sense of Definition~\ref{def:markovModel}.

Now, considering Definition \ref{def:markovModel} and Lemma \ref{lem:positiveRateExponentialSubset}, we can obtain the Markov matrices associated to a model by computing $M=e^{Qt}$ for each $t\geq 0$ and $Q \in \zerocolsum^+$. 
Previous work in \cite{sumnerLie} explored the connection between a set of zero row sum matrices $\zerocolsum$ forming a Lie algebra and the multiplicative closure of the set of corresponding Markov matrices $e^{\zerocolsum^+}$. Specifically, $\zerocolsum$ forms a Lie algebra if and only if $e^{\zerocolsum^+}$ is locally multiplicatively closed \cite{sumnerClosed}.

We now define and provide the appropriate abstract definition for the main algebraic structure of interest for the present article:
\begin{definition}
\label{def:jordanAlgebra}
A \emph{Jordan algebra} $\mathcal{J}$ is a vector space equipped with the bilinear product $\odot$, satisfying, for all $x,y \in \mathcal{J}$:
\begin{enumerate}
    \item $ x \odot y \in \mathcal{J} $
    \item $ x \odot y = y \odot x $
    \item $ (x \odot y)\odot(x \odot x) = x \odot (y \odot (x \odot x))$
\end{enumerate}

\end{definition}

We refer to a \emph{matrix algebra} as any subspace of $n\times n$ matrices that is closed under standard matrix multiplication.
With this in mind, we have the following easily understood realisation of the Jordan algebra concept:

\begin{lemma}
\label{lem:jordanProduct}
A subspace $\mathcal{J}$ of $n \times n$ matrices forms a Jordan algebra if it is closed under the bilinear product, defined for all $A,B \in \mathcal{J}$, as:
\[
A \odot B := A B  + BA
\]

\end{lemma}
\begin{proof}
The proof is a straightforward check against the conditions in Definition~\ref{def:jordanAlgebra}.
\end{proof}

Analogously, we recall that a (matrix) \emph{Lie algebra} of $n\times n$ matrices is similarly obtained by using the corresponding antisymmetric bilinear product:
\[
[A,B]:=AB-BA,
\]
and the following lemma provides a link between Lie and Jordan algebras of $n \times n$ matrices under these definitions.

\begin{lemma}
\label{lem:matrixIffJordanAndLie}
A set of matrices forms a matrix algebra if and only if it is both a Jordan algebra and a Lie algebra.
\end{lemma}
\begin{proof}

Suppose $\mathcal{A}$ forms both a Jordan algebra and a Lie algebra. Then, for all $A,B \in \mathcal{A}$:
\[
AB = \frac12 (AB +BA)+\frac12 (AB-BA) = \frac12 A \odot B + \frac12 [A,B] \in \mathcal{A},
\]
and hence $\mathcal{A}$ forms a matrix algebra. 

The other direction of the proof is similarly elementary.
\end{proof}

As a consequence of Fact~\ref{fact:rateSubAlgebra} and Lemma \ref{lem:matrixIffJordanAndLie} we have:
\begin{fact}
The zero row sum matrices $\zerocolsum_n$ form both a Lie algebra and a Jordan algebra.
\end{fact}

The Lie algebra structure of $\zerocolsum_n$ is described in detail in \cite{sumnerLie}.
In order to give an analogous convenient description of the Jordan algebra structure, we consider the `elementary' rate matrices:
\begin{definition}
\label{def:elementaryRateMatrix}
An \emph{elementary rate matrix} $L_{ij}\in \zerocolsum_n$ has $1$ in the $ij^{th}$ position, a $-1$ in the $ii^{th}$ position, and zeroes elsewhere. 
For convenience, we set $L_{ii}=0$ for each $i\in [n]$.  
\end{definition}

The set $\mathcal{B}=\{L_{ij} \in \mat_n(\mathbb{C}):1 \leq i\neq j \leq n \}$ forms a basis for $\zerocolsum_{n}$ as a vector space and, referring to Fact~\ref{fact:ratematricesconical}, $\zerocolsum_{n}^+$ is equal to the set of all conical combinations of the elementary rate matrices.
Of course, each $L_{ii}$ is equal to the zero matrix and not included in this statement. 
However, they are formally useful for describing the Jordan algebra structure of $\zerocolsum_{n}$ which is to follow.

We note that, as the Jordan product is bilinear, all Jordan products of matrices in $\zerocolsum_{n}$ can be reduced down to Jordan products of the elementary rate matrices.
Then, without too much trouble, one confirms:
\begin{lemma}
\label{lem:elementaryJordan}
For all choices $i,j\in [n]$:
\[
L_{ij}\odot L_{kl} = -\delta_{jl}(L_{ij}+L_{kl})+\delta_{jk}(L_{il}-L_{kl})+\delta_{il}(L_{kj}-L_{ij}),
\]
where $\delta_{ii}=1$ and $\delta_{ij}=0$ whenever $i\neq j$.
\end{lemma}

We are now ready to define a Jordan/Lie-Markov model, noting that, due to Lemma \ref{lem:matrixIffJordanAndLie}, it is possible for a Markov model to simultaneously be a Lie- and a Jordan-Markov model. 

\begin{definition}
\label{def:jordanMarkovModel}
Given a subset $\zerocolsum \subseteq \zerocolsum_{n}$ such that $\zerocolsum$ forms a Jordan/Lie algebra, we refer to $\zerocolsum^+ = \zerocolsum \ \cap \ \zerocolsum_{n}^+$ as a \emph{Jordan/Lie-Markov Model}.

\end{definition}

The following series of results provide our main motivation for considering the concept of Jordan-Markov models.
In particular, we see that a Markov model satisfying the Jordan property produces Markov matrices that are of the same form as the input rate matrices.

\begin{lemma}
\label{lem:jordanIffClosedUnderPowers}
A subspace $\mathcal{J} \subseteq \mat_n(\RR)$ forms a Jordan algebra if and only if it is closed under positive integer powers. 
That is,  
\[
A \odot B = AB+BA \in \mathcal{J}, \ \forall A, B \in \mathcal{J}\iff C^k \in \mathcal{J},\ \forall k \in \mathbb{N}, \ C \in \mathcal{J}.
\]
\end{lemma}
\begin{proof}
Suppose $\mathcal{J}$ is a Jordan algebra. Then $A^2 = \frac{1}{2} A\odot A \in \mathcal{J}$.
Then, supposing $A^{k-1}\in \mathcal{J}$ for integer $k>1$, we have
\[
\frac12 A^{k-1} \odot A =\frac12( A^{k-1}A+AA^{k-1} )=\frac12( A^k+A^k )=A^k \in \mathcal{J}.
\]
Hence, inductively, $\mathcal{J}$ is closed under positive integer powers.

On the other hand, suppose $\mathcal{J}$ is a matrix vector space and closed under powers.
Then for all $A, B \in \mathcal{J}$, we have $A+B,(A+B)^2 \in \mathcal{J}$.
Hence
$$(A+B)^2-A^2-B^2 = AB+BA=A\odot B \in \mathcal{J}, $$
and $\mathcal{J}$ is a Jordan algebra, as required.
\end{proof}

\begin{definition}
\label{def:minimalSubspace}
Suppose $\zerocolsum^+=\zerocolsum\cap \zerocolsum^+_n$ is a linear Markov model, so $\zerocolsum\subseteq \zerocolsum_n$ is a linear subspace. 
We say $\zerocolsum$ is \emph{minimal} if, for any other subspace $\zerocolsum'\subseteq\zerocolsum_{n}$ satisfying $\zerocolsum^+=\zerocolsum' \cap \zerocolsum_{n}^+$, it follows that $\zerocolsum \subseteq \zerocolsum'$.
\end{definition}
We note that a Jordan Algebra $\zerocolsum$ satisfying $\zerocolsum=\spam_{\RR}\left(\zerocolsum^+\right)$ is equivalent to $\zerocolsum$ having a `stochastic basis', as introduced in \cite{sumnerLie}.
\begin{lemma}
\label{lem:ifSubspaceMinimalThenEqualToSpan}
Suppose $\zerocolsum^+=\zerocolsum\cap \zerocolsum^+_n$ is a linear Markov model with $\zerocolsum$ minimal.
Then $\zerocolsum=\spam_{\RR}(\zerocolsum^+)$. 
In particular, the minimal subspace defining a linear Markov model is uniquely defined.
\end{lemma}
\begin{proof}
We have $\zerocolsum^+\subseteq  \spam_{\RR}(\zerocolsum^+)\subseteq \zerocolsum$ and $\zerocolsum^+\subseteq \zerocolsum_n^+$.
Thus,
\[
\zerocolsum \cap \zerocolsum_{n}^+=\zerocolsum^+\subseteq \spam_{\RR}(\zerocolsum^+) \cap \zerocolsum_{n}^+ \subseteq \zerocolsum \cap \zerocolsum_{n}^+
\]

Thus $$\spam_{\RR}(\zerocolsum^+) \cap \zerocolsum_{n}^+ = \zerocolsum \cap \zerocolsum_{n}^+$$ and the minimality of $\zerocolsum$ yields the result.
\end{proof}
\begin{example}
\label{exam:nonMinimalModel}
To illustrate a linear Markov model defined using a non-minimal subspace, consider the subspace,
\[
\zerocolsum:=
\left\{
\begin{pmatrixT}
0 & \alpha & \alpha \\
\beta & -2\alpha & \alpha \\
-\beta & \alpha &  -2\alpha
\end{pmatrixT}
:
\alpha,\beta\in \RR.
\right\}\subset \zerocolsum_3,
\]
and note that
\[
\zerocolsum^+:=
\zerocolsum\cap \zerocolsum_3^+
=
\left\{
\begin{pmatrixT}
0 & \alpha & \alpha \\
0 & -2\alpha & \alpha \\
0 & \alpha &  -2\alpha
\end{pmatrixT}
:
\alpha\geq 0
\right\}=
\spam_\RR(\zerocolsum^+)\cap \zerocolsum_3^+
, 
\]
with
\[
\spam_\RR(\zerocolsum^+)=
\left\{
\begin{pmatrixT}
0 & \alpha & \alpha \\
0 & -2\alpha & \alpha \\
0 & \alpha &  -2\alpha
\end{pmatrixT}
:
\alpha\in \RR
\right\}\neq \zerocolsum.
\]
\end{example}
The particular importance of Lemma~\ref{lem:ifSubspaceMinimalThenEqualToSpan} is demonstrated by the following: 

\begin{fact}
Suppose $\zerocolsum^+$ is a linear Markov model.
Then $\zerocolsum^+=\spam_\RR(\zerocolsum^+) \cap \zerocolsum_{n}^+$.
As a consequence, any rate matrix $Q\in \zerocolsum_n^+$ that is expressible as a linear combination of rate matrices taken from $\zerocolsum^+$ again lies in $\zerocolsum^+$.
That is, a linear Markov model is closed under linear combinations of rate matrices that yield rate matrices.
\end{fact}

The following example illustrates a Markov model where, despite the appearance of a linear parameterisation, the implicit existence of inequalities on the matrix entries means the model is not linear.
\begin{example}
\label{exam:nonLinearModel}
Consider
\[
\zerocolsum^+:=
\left\{
\begin{pmatrixT}
-\alpha & \alpha+\beta \\
\alpha & -(\alpha+\beta)
\end{pmatrixT}
:
\alpha,\beta\geq 0.
\right\},
\]
and note that
\[
\begin{pmatrixT}
-1 & 0 \\
1 & 0
\end{pmatrixT}
\notin 
\zerocolsum^+.
\]
Observe that $\spam_\RR(\zerocolsum^+)=\zerocolsum_2$ and hence $\spam_\RR(\zerocolsum^+)\cap \zerocolsum_2^+=\zerocolsum_2^+ \neq \zerocolsum^+$, we see that $\zerocolsum^+$ is not a linear Markov model.
\end{example}

The following definition is motivated by the observation presented in Lemma~\ref{lem:positiveRateExponentialSubset} and captures the particular property of Markov models that is the central point of study for the present article. 

\begin{definition}
\label{def:expclosed}
We say a Markov model $\zerocolsum^+$ is \emph{uniformization stable} if, for all $t\geq 0$ and $Q \in \zerocolsum^+$, there exists $\hat{Q} \in \zerocolsum^+$ such that $e^{Qt}-I_n = \hat{Q}$.
\end{definition}

The motivation for our terminology comes from the standard `uniformization' procedure, first discussed in \cite{grassman}: Suppose $Q\in \zerocolsum^+_n$ is a rate matrix and choose $\lambda>0$ such that $R=I_n+
\frac{1}{\lambda}Q$ is a Markov matrix.
Then,
\[
M=e^{Qt}=e^{-\lambda t}e^{\lambda Rt}=\sum_{k\geq 0}e^{-\lambda t}\frac{(\lambda t)^k}{k!}R^k,
\]
which decomposes the Markov chain into a Poisson process, with mean $\lambda$, and a discrete time Markov chain with transition matrix $R$.
An easy check then shows that Definition~\ref{def:expclosed} is equivalent to the requirement that the Markov matrices $M$ and $R$ are of the same form for all choices of $Q$ in the model. That is, if $\zerocolsum^+$ is uniformization stable, for any $Q\in \zerocolsum^+$, there are $Q_M, Q_R \in \zerocolsum^+$ such that $M=I_n+Q_M$ and $R=I_n+Q_R$.

\begin{theorem}
\label{thm:jordanIffExponential}
Suppose $\zerocolsum^+$ is a linear Markov model. 
Noting that $\zerocolsum^+=\spam_\RR(\zerocolsum^+)\cap \zerocolsum^+_n$, the following statements are equivalent:

\begin{enumerate}
    \item $\zerocolsum^+$ is uniformization stable;
    \item $\spam_\RR(\zerocolsum^+)$ forms a Jordan algebra.
\end{enumerate}

\end{theorem}
\begin{proof}
Suppose $\zerocolsum^+$ is a linear Markov model and (1.) holds.
Then, for all $Q \in \zerocolsum^+$ and $ t\geq 0$, we have $Qt \in \zerocolsum^+$ also and hence $\hat{Q}(t) := e^{Qt}-I_n \in \zerocolsum^+, \ \forall t \geq 0$. 
Observing that $\spam_\RR(\zerocolsum^+)$ is a real linear subspace of $\zerocolsum_n$ and is therefore topologically closed, we have for all $k\in \mathbb{N}$:

$$\left. \frac{d^k}{dt^k}\hat{Q}(t) \right|_{t=0} = \frac{d^k}{dt^k}\left.\left(- I_n+\sum_{i=0}^{\infty}\frac{(Qt)^i}{i!}\right)\right|_{t=0}
= Q^k \in \spam_\RR(\zerocolsum^+).$$ 

Since $\zerocolsum^+$ is a linear Markov model, for any $Q_1, Q_2 \in \zerocolsum^+$ we have $ Q_1+ Q_2 \in \zerocolsum^+$ (Fact \ref{fact:ratematricesconical}) and combining with the above yields,
\[
Q_1,Q_2\in \zerocolsum^+ \implies Q_1Q_2+Q_2Q_1 =(Q_1+ Q_2)^2- Q_1^2-Q_2^2 \in \spam_\RR(\zerocolsum^+).
\]
One then easily extends this result linearly to show $R,R'\in \spam_\RR(\zerocolsum^+)$ implies $R\odot R'=RR'+R'R\in \spam_\RR(\zerocolsum^+)$ also.
Thus $\spam_\RR(\zerocolsum^+)$ is a Jordan algebra, as required.

The converse is an easy application of Lemma~\ref{lem:jordanIffClosedUnderPowers} and that $\spam_\RR(\zerocolsum^+)$is a topologically closed subspace of $\zerocolsum_n$.

\end{proof}

\begin{example}
\label{exam:jordanUniformization}
As a demonstration of Theorem~\ref{thm:jordanIffExponential}, we consider the following examples.
\begin{enumerate}
    \item Consider the well-known phylogenetic model first discussed in \cite{felsenstein1981}, which falls into the family of `equal-input' models (as described in \cite[Chap 3.8.1]{steel} --- see more below).
To this end, consider the linear space
\[\text{EI}_4 := 
\left\{
\begin{pmatrixT} 
-(\beta+\gamma+\delta) & \alpha & \alpha & \alpha\\
\beta & -(\alpha+\gamma+\delta) & \beta & \beta\\
\gamma & \gamma & -(\alpha+\beta+\delta) & \gamma\\
\delta & \delta & \delta & -(\alpha+\beta+\gamma)
\end{pmatrixT}: \alpha,\beta,\gamma,\delta \in \RR\right\}\]
and define the linear Markov model $\text{EI}_4^+=\text{EI}_4\cap \zerocolsum^+$.

It is straightforward to show that for any $A,B\in \text{EI}_4$, we have the matrix product $AB\in \text{EI}_4$ also, so we conclude from Lemma~\ref{lem:matrixIffJordanAndLie} that $\text{EI}_4$ forms a Jordan algebra and hence $\text{EI}_4^+$ is uniformization stable.

In fact, taking 
\[
Q=
\begin{pmatrixT} -(\beta+\gamma+\delta) & \alpha & \alpha & \alpha \\
\beta & -(\alpha+\gamma+\delta) & \beta & \beta \\
\gamma & \gamma & -(\alpha+\beta+\delta) & \gamma \\
\delta & \delta & \delta & -(\alpha+\beta+\gamma) 
\end{pmatrixT}
\in \text{EI}_4^+,
\]
we have $Q^2=-\lambda Q$ with $\lambda=\alpha+\beta+\gamma+\delta$, so $Q^k=(-1)^{k-1}\lambda^{k-1} Q$ and, for $t\geq 0$, we find
\[
e^{Qt}-I_n=\frac{1-e^{-\lambda t}}{\lambda } Q\in \text{EI}_4^+,
\]
so the uniformization stability of this model is clear.

\item In \cite{draisma} the mathematically useful class of `equivariant' models was defined, as follows.
Given a permutation group $G\leq S_n$ the $G$-equivariant model is defined to consist of all $Q\in \zerocolsum^+$ satisfying $Q=K^T_\sigma QK_\sigma$ for all $\sigma\in G$.
It is a straightforward exercise to show that for each $G\leq S_n$ the corresponding equivariant model forms a matrix algebra and is hence uniformization stable (Lemma \ref{lem:matrixIffJordanAndLie}).
For $n=4$, a notable example of an equivariant model is the so-called Kimura 3 parameter model \cite{Kimura_1981} which occurs with $G=\{e,(12)(34),(13)(24),(14)(23)\}$. We extend the concept of equivariant models to time-reversible models in Section \ref{sec:timeReversible}.
\item Another important class of Markov models forming matrix algebras are the so-called `group-based' models, where the underlying states in the Markov chain are identified with a finite group \cite[Chapter~7.32]{steel}, as well as their natural extension to models where the states are drawn from a finite semigroup \cite{woodhams}.
Since any model drawn from these classes forms a matrix algebra \cite{woodhams}, it follows each of these models is uniformization stable (Lemma~\ref{lem:matrixIffJordanAndLie}).
\end{enumerate}
\end{example}

We provide a non-trivial example of a model that is not uniformization stable at the end of the next section.

\section{Jordan algebra structure of time-reversible models}
\label{sec:timeReversible}

Since the earliest introduction of the maximum likelihood approach to phylogenetic tree inference \cite{felsenstein1981}, the use of time-reversible Markov models has been almost ubiquitous in phylogenetics and forms a basis for the dominant hierarchy of models available in contemporary phylogenetic inference software such as the IQ-TREE package \cite{minh}. 
In this section we show that some (but not all) time-reversible Markov models are uniformization stable and, in doing so, reveal their underlying Jordan algebraic structure. We note that the general time-reversible model forms an algebraic variety in $\text{Mat}_n(\RR)$, as it is defined by cubic polynomial constraints, specifically the Kolmogorov criterion \cite[Chapter 4]{pollett}.
We begin the discussion by drawing attention to the origin of the Jordan algebraic structure from a more general perspective.

To this end, consider a fixed matrix $D\in\mat_n(\RR)$ and recall that the \emph{commutant},
\[
\text{Comm}(D):=\{A\in \mat_n(\RR):DA=AD \},
\]
forms a matrix subalgebra of $ \mat_n(\RR)$.
This construction has various important applications in linear algebra; perhaps most importantly: $D$ is \emph{cyclic}, or \emph{simple} (has distinct eigenvalues) in the diagonalisable case, if and only if $\text{Comm}(D)=\RR[D]:=\spam_{\RR}(I,D,D^2,D^3,\ldots)$. 

In analogy, consider:
\begin{lemma}
\label{lem:commT}
For fixed $D\in \mat_n(\RR)$, the set 
\[
\text{Comm}^T(D):=\left\{A\in \mat_n(\RR):DA=A^TD \right\}
\]
forms a Jordan subalgebra of $\mat_n(\RR)$.
\end{lemma}
\begin{proof}
The condition for membership in $\text{Comm}^T(D)$ is linear due to the linearity of the matrix transpose. 
Now suppose $A,B\in \text{Comm}^T(D)$ and consider
\[
D(A\odot B)=D(AB+BA)=(A^TB^T+B^TA^T)D=(BA+AB)^TD=(A\odot B)^TD.
\]
Therefore $\text{Comm}^T(D)$ forms a Jordan algebra, as required.
\end{proof}

We now use Lemma~\ref{lem:commT} to reveal the underlying Jordan structure in time-reversible Markov models. 
To avoid trivialities, throughout this section we will assume $\pi=(\pi_1,\pi_2,\ldots,\pi_n)$ is a strictly positive distribution vector; so each $\pi_i>0$ and $\sum_{i=1}^n\pi_i=1$. 
A \emph{time-reversible} Markov chain is then obtained by taking a rate matrix $Q=(q_{ij})\in \zerocolsum^+_n$ satisfying the detailed balance equations:
\[
\pi_i q_{ij}=\pi_j q_{ji},
\]
for all $i,j\in [n]$.
In this scenario, one readily shows $\pi Q=0$ so $\pi$ is an equilibrium distribution for $Q$.

Now letting $D(\pi)$ be the $n\times n$ diagonal matrix with diagonal entries $\pi_i$, an easy check shows that the detailed balance conditions are equivalent to the matrix equation,
\[
D(\pi)Q=Q^TD(\pi).
\]

For each fixed distribution vector $\pi$, it is then natural to define, as in \cite{tavare}, the \emph{general time-reversible model} $\text{GTR}_\pi^+:=\text{GTR}_\pi\cap \zerocolsum^+_n$ where
\[
\text{GTR}_\pi:=\left\{Q\in \zerocolsum_n:D(\pi)Q=Q^TD(\pi)\right\}\subset \zerocolsum_n.
\]
We obtain:
\begin{lemma}
\label{lem:GTRJordan}
For each distribution vector $\pi$, the set $\text{GTR}_\pi$ forms a Jordan algebra.
\end{lemma}
\begin{proof}
The result follows by applying Lemma~\ref{lem:commT} and expressing $\text{GTR}_\pi$ as the intersection of two Jordan algebras:
\[
\text{GTR}_\pi=\text{Comm}^T(D(\pi))\cap \zerocolsum_n=\left\{Q\in \mat_n(\RR):D(\pi)Q=Q^TD(\pi)\right\}\cap \zerocolsum_n.
\]

\end{proof}

\begin{theorem}
\label{thm:GTRExponential}
For each distribution vector $\pi$, the time-reversible model $\text{GTR}_\pi^+$ is uniformization stable. 
\end{theorem}
\begin{proof}

The result is an easy application of Lemma~\ref{lem:GTRJordan} and Theorem~\ref{thm:jordanIffExponential}. Note that $\text{GTR}_{\pi}$ is minimal, in the sense of Definition \ref{def:minimalSubspace}, and hence Theorem \ref{thm:jordanIffExponential} applies here. To see this, an easy check confirms that each $Q = \left( q_{ij} \right)\in \text{GTR}_{\pi}$ can be written as $Q= \sum_{ij}\frac{q_{ij}}{\pi_j} \hat{L}_{ij}$, where each $\hat{L}_{ij}:=\pi_jL_{ij}+\pi_iL_{ji}$ is in $\text{GTR}_{\pi}^+$. Hence $\text{GTR}_{\pi}=\spam_{\RR}\left(\text{GTR}_{\pi}^+\right)$.
\end{proof}

We now state the following fact that is possibly well-known but is usually not explicitly noted:
\begin{corollary}
The \emph{general time-reversible model} is a uniformization stable Markov model and, for $n>2$, is non-linear.
\end{corollary}
\begin{proof}
The general time-reversible model can be expressed as a union parameterised over distribution vectors:
\[
\text{GTR}^+=\cup_{\pi} \text{GTR}^+_\pi=\left(\cup_{\pi} \text{GTR}_\pi\right)\cap \zerocolsum_n^+,
\]
(where it should be noted that this is not a set partition since, in particular, the zero matrix occurs in every $\text{GTR}^+_\pi$, together with further non-trivial cases). 
Since each $\text{GTR}^+_\pi$ is uniformization stable, so is $\text{GTR}^+$.
Note that for $n=2$, $\text{GTR}^+$ is linear, as any 2 state Markov chain is time-reversible. That is, $\text{GTR}^+ = \zerocolsum_2^+$, which also holds for $n=1$. We establish that $\text{GTR}^+$ is not a linear model for $n>2$. Consider,
\[
Q
=
\begin{pmatrixT}
-3 & 1 & 1 \\
1 & -3 & 1\\
2 & 2 & -2
\end{pmatrixT}
\in \text{GTR}_{\pi},
\qquad 
Q'=
\begin{pmatrixT}
-1 & 0 & 1 \\
0 & 0 & 0\\
1 & 0 & -1
\end{pmatrixT}
\in \text{GTR}_{\pi'}.
\]
with $\pi=(1/4,1/4,1/2)$ and $\pi'=(1/3,1/3,1/3)$.
A short computation establishes that $\widehat{Q}=Q+Q'$ has the (unique) stationary distribution vector $\widehat{\pi}=\textstyle{\frac{1}{24}}(7,6,11)$, however $D(\widehat{\pi})\widehat{Q}\neq \widehat{Q}^TD(\widehat{\pi})$ and hence $\widehat{Q}\notin \text{GTR}^+$. Considering the $3 \times 3$ case is sufficient as we can take two $n \times n$ matrices $R,R'$ with $Q$ and $Q'$ embedded in their top left corners, respectively, and zero entries elsewhere. The stationary distribution $\hat{\pi}_R$ of $\hat{R}=R+R'$ is not unique for $n>3$ but the first three entries are the same as $\hat{\pi}$ (up to scaling), while the remaining entries can take any value. Hence we observe $D\left(\hat{\pi}_R\right)\hat{R} \neq \hat{R}^TD\left(\hat{\pi}_R\right)$ in general.
\end{proof}

\subsection{Equivariant time-reversible models}

Further to the above discussion, we show that some other models used for model testing in phylogenetics from the so-called `GTR hierarchy' \cite[Chapter 13]{felsenstein} are also uniformization stable. Rather than taking time-reversible models with arbitrary constraints, we define the concept of equivariant time-reversible models and study the resulting hierarchy of models for uniformization stability. We initially demonstrate this through example with the so-called `Tamura-Nei 1993' model \cite{tamuraNei} (hereby referred to as TN), a well-known model of DNA evolution in phylogenetic modelling. This model can be understood as a combination of the four-state equal-input model, together with 2 additional parameters to specify the relative rates of two distinct transitions (substitutions between the purines $A$ and $G$, and between the pyrimidines $C$ and $T$) and transversions (all other substitutions). 
From here on we use the ordering $A,G,C,T$ of the nucleotides as labels on the rows and columns of matrices. We now define the Tamura-Nei Model $\text{TN}_{\pi}^+$ for a fixed stationary distribution vector $\pi$, and note that, as for the GTR model, we can define the full model $\text{TN}^+$ as the union of the $\text{TN}_{\pi}^+$ over all possible distribution vectors $\pi$. Throughout this section for simplicity we take $\pi=\left(\pi_1,\pi_2,\pi_3,\pi_4\right)$ fixed and generic. That is, the only constraints we allow on $\pi$ is that each $\pi_i$ is positive and $\sum_{i}\pi_i=1$. 

Specifically, we have
\[
\text{TN}^+_\pi=
\left\{
\begin{pmatrix}
\ast & \pi_2\kappa_1 & \lambda\pi_3 & \lambda\pi_4 \\
\pi_1\kappa_1 & \ast & \lambda\pi_3 & \lambda\pi_4 \\
\lambda\pi_1 & \lambda\pi_2 & \ast & \pi_4\kappa_2\\
\lambda\pi_1 & \lambda\pi_2 & \pi_3\kappa_2 & \ast
\end{pmatrix}:
\kappa_1,\kappa_2,\lambda\geq 0
\right\},
\]
where, for simplicity, we use notation where the diagonal entries of matrices in $\zerocolsum_n$ are labelled $\ast$, since these values are determined by the row sum condition. Note that each $\text{TN}_{\pi}:=\spam_{\RR}\left(\text{TN}_{\pi}^+\right)$ forms a linear space and is more specifically a linear subspace of $\text{GTR}_{\pi}$.

We now show that $\text{TN}_{\pi}^+$ is uniformization stable for each choice of distribution vector $\pi$. We do this by first constructing a basis for the linear space $\text{TN}_{\pi}$ and demonstrating that all Jordan products of these basis matrices are contained in $\text{TN}_{\pi}$. Consider an arbitrary matrix $Q\in \text{TN}_{\pi}$, where $\pi=\left(\pi_1,\pi_2,\pi_3,\pi_4\right)$, and observe,

\begin{align*}
Q=\begin{pmatrix}
    \ast & \pi_2 \kappa_1 & \lambda\pi_3& \lambda\pi_4 \\
    \pi_1 \kappa_1 & \ast & \lambda\pi_3 & \lambda\pi_4 \\
    \lambda\pi_1 & \lambda\pi_2 & \ast & \pi_4 \kappa_2 \\
    \lambda\pi_1 & \lambda\pi_2  & \pi_3 \kappa_2 & \ast 
\end{pmatrix} &= \kappa_1 \begin{pmatrix}
    -\pi_2 & \pi_2 & 0 & 0 \\
    \pi_1 & - \pi_1 & 0 & 0 \\
    0 & 0 & 0 &  0 \\
    0 & 0  & 0 & 0
\end{pmatrix} + \kappa_2 \begin{pmatrix}
    0 & 0 & 0 & 0\\
    0 & 0 & 0 & 0 \\
    0 & 0 & -\pi_4 & \pi_4 \\
    0 & 0  & \pi_3 & -\pi_3 
\end{pmatrix}\\&+\lambda\begin{pmatrix}
    -(\pi_3+\pi_4) & 0 & \pi_3 & \pi_4 \\
    0 & -(\pi_3+\pi_4) & \pi_3 & \pi_4 \\
    \pi_1 & \pi_2 & -(\pi_1+\pi_2) & 0 \\
    \pi_1 & \pi_2  & 0 & -(\pi_1+\pi_2) 
\end{pmatrix}\\
&= \kappa_1 A + \kappa_2 B + \lambda C.
\end{align*}
It is clear that the matrices $A,B$ and $C$ form a basis for the linear space $\text{TN}_{\pi}$. We then have to check that the Jordan products of these basis matrices are contained in $\text{TN}_{\pi}$. We first check the squares of each matrix. One can confirm that $A^2=-(\pi_1+\pi_2)A$, $B^2=-(\pi_3+\pi_4)B$, and $C^2=(\pi_3+\pi_4)A+(\pi_1+\pi_2)B-C$.
Additionally, we have $AB=BA=0$, $AC=CA= -(\pi_3+\pi_4)A$, and $BC=CB = (-\pi_1+\pi_2)B$. Hence the Jordan products $A\odot A$, $B\odot B$, $C \odot C$, $A\odot B$, $A \odot C$, and $B \odot C$ are all contained in $\text{TN}_{\pi}$.
\begin{theorem}
The Tamura-Nei 1993 model $\text{TN}^+$ is uniformization stable.

\end{theorem}
\begin{proof}
Each $\text{TN}_{\pi}$ is minimal in the sense of Definiton \ref{def:minimalSubspace}, hence by the above discussion and Theorem \ref{thm:jordanIffExponential} each $\text{TN}_{\pi}^+$ is uniformization stable. As $\text{TN}^+$ is the union $\text{TN}^+=\cup_{\pi} \text{TN}^+_{\pi}$, we see that $\text{TN}^+$ is uniformization stable.
\end{proof}

We now observe that we can exploit the natural symmetry of the rate matrices in $\text{TN}_{\pi}$ to define this model in algebraic terms. Further, this definition can be generalised to define a hierarchy of time-reversible models (see Definition \ref{def:equitimerev}). Before this, we first recall a standard paramaterisation of time-reversible matrices that factorises a general $Q \in \text{GTR}_{\pi}$ into two symmetric matrices $D(\pi)$ and $T$, expressed as $Q=TD(\pi)$, where $D(\pi)$ is as defined in the previous section in the definiton of the $\text{GTR}$ model \cite{lanave}. We modify this standard form slightly by extracting the diagonal entries of $Q$ prior to the factorisation. This gives the general form,
$$Q=SD(\pi) + \text{diag}(Q).$$
Here $\text{diag}(Q)$ is a diagonal matrix containing the diagonal elements of $Q$, and the $S$ matrix is a symmetric matrix that contains the `non-$\pi$' parameters of $Q$. For example, for a general $Q \in \text{TN}_{\pi}$ with $\pi=(\pi_1,\pi_2,\pi_3,\pi_4)$,
$$Q=\begin{pmatrix} 0 & \kappa_1 & \lambda & \lambda \\
\kappa_1 & 0 & \lambda & \lambda \\ 
\lambda & \lambda & 0 & \kappa_2 \\ 
\lambda & \lambda& \kappa_2 & 0\end{pmatrix}\begin{pmatrix} \pi_1 & 0& 0& 0 \\ 
0& \pi_2 & 0 &0\\
0 & 0 &\pi_3 & 0\\
0 & 0 & 0 & \pi_4 \end{pmatrix}+ \text{diag}(Q).$$
As $Q$ is a rate matrix, we note that the diagonal matrix $\text{diag}(Q)$ is completely defined by the entries in $S$ and $D(\pi)$. That is, $\text{diag}(Q)$ is not dependent on any additional parameters.

Keeping in mind the $A,G,C,T$ ordering of states that we are assuming, consider the permutation group $G_{\text{TN}}$ as a subgroup of the permutation group on 4 elements, $S_4$:
$$G_{\text{TN}}=\{e,(12),(34),(12)(34)\},$$ and associate to each permutation $\sigma\in G_{\text{TN}}$ with a $4\times 4$ permutation matrix $K_\sigma$. Each permutation matrix $K_{\sigma}$ is defined as the matrix satisfying $e_iK_{\sigma} := e_{\sigma(i)}$ for all $i \in [n]$, where $e_i$ denotes the row vector containing 1 in the i'th position and 0's elsewhere.

One can see that it is the $S$ matrices in the parameterisation of each $Q$ that has symmetry under $G_{\text{TN}}$, rather than the entire $Q$ matrix. 
That is, for some $Q=SD(\pi)+\text{diag}(Q)\in \text{TN}_\pi$, the action of some $\sigma \in G_{\text{TN}}$ gives,
\begin{align*}
K_{\sigma}^{-1} SD(\pi)K_{\sigma}& = K_{\sigma}^{-1}( SD(\pi))K_{\sigma}\\
&= K_{\sigma}^{-1}SK_{\sigma}D(\pi K_{\sigma}) \\
&= SD(\pi K_{\sigma})
\end{align*}

Note that here we do not need to worry about $K_{\sigma}^{-1}\text{diag}(Q)K_{\sigma}$ as $\zerocolsum_n$ is closed under the action of $S_n$. That is, the entries of $K_{\sigma}^{-1}\text{diag}(Q)K_{\sigma}$ will be the diagonal elements required to make $K_{\sigma}^{-1}QK_{\sigma}=K_{\sigma}^{-1} SD(\pi)K_{\sigma}+K_{\sigma}^{-1}\text{diag}(Q)K_{\sigma}$ a zero row sum matrix.
Hence we have,
$$K_{\sigma}^{-1}QK_{\sigma}= S D(\pi K_{\sigma}) + \text{diag}(K_{\sigma}^{-1}QK_{\sigma}) \in \text{TN}_{\pi K_{\sigma}},$$
meaning that $G_{\text{TN}}$ does not give a group action on the individual $\text{TN}_\pi$ but rather acts on the entire $\text{TN}$ model, sending $Q\in \text{TN}^+_\pi$ to some $Q'\in \text{TN}^+_{\pi K_\sigma}$.

Considering the above parameterisation $Q=SD(\pi)+\text{diag}(Q)$ of matrices $Q\in\text{GTR}$ one can see that each $Q$ in GTR can be completely defined as a function of two vectors: the stationary distribution $\pi$, and a vector containing the elements of the associated $S$ matrix that we denote by $s$. More specifically, let $S=(s_{ij})$, and define $s:=(s_{12},s_{13},s_{14},s_{23},s_{24},s_{34})$. Hence each $Q\in \text{GTR}_{\pi}$ can be written $Q\equiv \mathcal{Q}(s,\pi)$, that is, $Q=\left(s_{ij} \pi_i\right)$. From here on we write $\mathcal{Q}(s,\pi)$ to refer to the time-reversible matrix that is defined by the vectors $s$ and $\pi$ in the above way. For example, matrices from $\text{TN}_{\pi}$ can be written $\mathcal{Q}(s,\pi)$ with $s=(\kappa_1,\lambda,\lambda,\lambda,\lambda,\kappa_2)$ and $\kappa_1,\kappa_2,\lambda \in \RR$.
With this notation, the equivariant condition $K^{-1}SK=S$ is equivalent to $K^{-1}\mathcal{Q}(s,\pi)K=\mathcal{Q}(s,\pi K)$. That is, the $s$ vector is unchanged, but the $\pi$ vector is permuted according to $K$.  
To explicitly illustrate the action of the permutations on the $\text{TN}$
rate matrices, consider 
\[
\mathcal{Q}(s,\pi)=
\begin{pmatrixT}
\ast & \pi_1 \kappa_1 & \lambda \pi_1 & \lambda \pi_1 \\
\pi_2 \kappa_1 & \ast & \lambda\pi_2 & \lambda\pi_2 \\
\lambda \pi_3 & \lambda\pi_3 & \ast & \pi_3 \kappa_2\\
\lambda\pi_4 & \lambda\pi_4 & \pi_4 \kappa_2 & \ast
\end{pmatrixT}\in \text{TN}_{\pi}
\]
with $s=(\kappa_1,\lambda,\lambda,\lambda,\lambda,\kappa_2)$ and $\pi=(\pi_1,\pi_2,\pi_3,\pi_4)$.
Taking $\sigma=(12)(34)\in G_{\text{TN}}$ and $K\equiv K_\sigma$, we find:

\[
K^{-1} \mathcal{Q}(s,\pi) K
=
\begin{pmatrix}
0 & 0 & 0 & 1 \\
0 & 0 & 1 & 0 \\
1 & 0 & 0 & 0\\
0 & 1 & 0 & 0
\end{pmatrix}
\begin{pmatrixT}
\ast & \pi_1\kappa_1 & \lambda\pi_1 & \lambda\pi_1 \\
\pi_2\kappa_1 & \ast & \lambda\pi_2 & \lambda\pi_2 \\
\lambda\pi_3 & \lambda\pi_3 & \ast & \pi_3 \kappa_2\\
\lambda\pi_4 & \lambda\pi_4 & \pi_4\kappa_2 & \ast
\end{pmatrixT}
\begin{pmatrix}
0 & 0 & 1 & 0 \\
0 & 0 & 0 & 1 \\
0 & 1 & 0 & 0\\
1 & 0 & 0 & 0
\end{pmatrix}
=
\begin{pmatrixT}
\ast & \pi_2 \kappa_1 & \lambda\pi_2 & \lambda\pi_2 \\
\pi_1 \kappa_1 & \ast & \lambda\pi_1 & \lambda\pi_1 \\
\lambda\pi_4 & \lambda\pi_4 & \ast & \pi_4 \kappa_2\\
\lambda\pi_3 & \lambda\pi_3 & \pi_3 \kappa_2 & \ast
\end{pmatrixT}
\]
Clearly the $s$ vector is unchanged, 
 but the $\pi$ vector is permuted according to the choice of $\sigma$, that is $K^{-1} \mathcal{Q}(s,\pi) K=\mathcal{Q}(s,\pi K)$.

We can hence define each $\text{TN}_{\pi}$ as follows,

$$\text{TN}_{\pi}:= \{\mathcal{Q}(s,\pi) \in \text{GTR}_{\pi}: K_{\sigma}^{-1}\mathcal{Q}(s,\pi)K_{\sigma}=\mathcal{Q}(s,\pi K_{\sigma}), \forall \sigma \in G_{\text{TN}}\}.$$
A quick check shows that this definition is equivalent to $\text{TN}_{\pi}=\spam_{\RR}(\text{TN}_{\pi}^+)$ where $\text{TN}_{\pi}^+$ is as defined at the beginning of this section.

We can now easily generalise this definition to an arbitrary group $G\leq S_n$ to define a hierarchy of $n$-state time-reversible models that exhibit $G$-symmetry. 
\begin{definition}
\label{def:equitimerev}
Given a permutation subgroup $G\leq S_n$, the $G$-\emph{equivariant, time-reversible} model $\text{TR}_G^+$ is defined as,

\[
\text{TR}_G^+=\cup_{\pi} \text{TR}^+_{(\pi,G)}=\left(\cup_{\pi} \text{TR}_{(\pi,G)}\right)\cap \zerocolsum_n^+,\] where,
\[\text{TR}_{(\pi,G)}:= \{\mathcal{Q}(s,\pi) \in \text{GTR}_{\pi}: K_{\sigma}^{-1}\mathcal{Q}(s,\pi)K_{\sigma}=\mathcal{Q}(s,\pi K_{\sigma}), \forall \sigma \in G\}.\]
\end{definition}
The linear space $\text{TR}_{(\pi,G)}$ is minimal in the sense of Definition \ref{def:minimalSubspace}, that is, $\text{TR}_{(\pi,G)}=\spam_{\RR}(\text{TR}_{(\pi,G)}^+)$. The argument is a generalisation of the argument for the minimality of GTR (see the proof of Theorem \ref{thm:GTRExponential}), except that the relevant basis matrices will be conical combinations of the $\hat{L}_{ij}$ matrices defined in that proof.

We pause here to note that Draisma and Kuttler \cite[Example 3.1]{draisma} give an incorrect formulation of some time-reversible models in the context of their definition of equivariant Markov models. Specifically, they incorrectly characterise time-reversible models by defining them as a collection of rate matrices starting with a distribution $\pi$ on the `root' of the phylogenetic tree.
Additionally, they incorrectly characterise the symmetries of the well-known HKY model \cite{hky} by claiming that HKY is an equivariant model defined by symmetry imposed by the group $G_{\text{TN}}$.
 However, under our definition of $G$-equivariant time-reversible model, setting $G$ to be the dihedral group $D_4$ captures the HKY model. That is, $\text{TR}_{D_4}^+ = \text{HKY}^+$. This is unfortunate, since the otherwise astute definition of equivariant models given in that paper is useful for algebraically characterising a large hierarchy of Markov models (most of which are not time-reversible) and has led to various applications and mathematical analysis \cite{casanellas}. 

In general, the $G$-equivariant, time-reversible models are not uniformization stable, and we demonstrate this presently using the HKY model as an example. But first we state a lemma that will be useful in the following discussion.
\begin{lemma}
\label{lem:TRUnififfTRpiUnif}
For $G\leq S_n$, $\text{TR}_{G}^+$ is uniformization stable if and only if $\text{TR}_{(\pi,G)}^+$ is uniformization stable for every distribution vector $\pi$.
\end{lemma}
\begin{proof}
Suppose that $\text{TR}_{G}^+$ is uniformization stable for some $G\leq S_n$. Then for some $Q\in \text{TR}_{(\pi,G)}$ with $\pi$ fixed, we have
\begin{align*}
D(\pi)(e^Q-I_n)&=\left( D(\pi)+D(\pi)Q +  \frac{D(\pi)Q^2}{2} + \ldots \right) - D(\pi)\\ &= \left( D(\pi)+Q^TD(\pi) +  \frac{(Q^T)^2D(\pi)}{2} + \ldots \right) - D(\pi) \\
&= (e^{Q^T}-I_n)D(\pi)\\
&=(e^Q-I_n)^TD(\pi),
\end{align*}
where the last line follows from the fact that the matrix exponential preserves the matrix transpose.
Hence $e^{Q}-I\in \text{GTR}_{\pi} \cap \text{TR}_G^+ = \text{TR}_{(\pi,G)}^+$. So $\text{TR}_{(\pi,G)}^+$ is uniformization stable. The other direction of the proof is clear by the definition of $\text{TR}_G^+$.
\end{proof}

We define,
\[
\text{HKY}^+_\pi:=
\left\{
\begin{pmatrix}
\ast & \pi_2\kappa & \lambda \pi_3 & \lambda\pi_4 \\
\pi_1\kappa & \ast & \lambda\pi_3 & \lambda\pi_4 \\
\lambda\pi_1 & \lambda\pi_2 & \ast & \pi_4\kappa\\
\lambda\pi_1 & \lambda\pi_2 & \pi_3\kappa & \ast
\end{pmatrix}:
\kappa,\lambda \geq 0
\right\},
\]

Equivalently, the HKY model is captured under the $G$-equivariant time-reversible model definition by setting $G$ equal to the dihedral group $D_4$,
\[
D_4=\{e,(12),(34),(12)(34),(13)(24),(14)(23),(1324),(1423)\}<S_4,
\]
That is, 

$$\text{HKY}_\pi =\text{TR}_{(\pi,D_4)} = \{\mathcal{Q}(s,\pi) \in \text{GTR}_{\pi}: K_{\sigma}^{-1}\mathcal{Q}(s,\pi)K_{\sigma}=\mathcal{Q}(s,\pi K_{\sigma}), \forall \sigma \in D_4\} $$

We will now show that $\text{HKY}_{\pi}$ does not form a Jordan algebra and hence the HKY model is not uniformization stable. We proceed by following the same procedure we used earlier to prove that $\text{TN}_{\pi}$ formed a Jordan algebra.
Consider an arbitrary rate matrix $\mathcal{Q}(s,\pi)\in \text{HKY}_{\pi}^+$. 
We can decompose this matrix into a linear combination of basis matrices. 

\begin{align*}
\mathcal{Q}(s,\pi)&=\begin{pmatrix}
\ast & \pi_2\kappa & \lambda \pi_3 & \lambda\pi_4 \\
\pi_1\kappa & \ast & \lambda\pi_3 & \lambda\pi_4 \\
\lambda\pi_1 & \lambda\pi_2 & \ast & \pi_4\kappa\\
\lambda\pi_1 & \lambda\pi_2 & \lambda\pi_3\kappa & \ast
\end{pmatrix} \\&= \kappa \begin{pmatrix}
    -\pi_2 & \pi_2 & 0 & 0 \\
    \pi_1 & - \pi_1 & 0 & 0 \\
    0 & 0& -\pi_4 & \pi_4 \\
    0 & 0  & \pi_3 & -\pi_3 
\end{pmatrix} +\lambda\begin{pmatrix}
    -(\pi_3+\pi_4) & 0 & \pi_3 & \pi_4 \\
    0 & -(\pi_3+\pi_4) & \pi_3 & \pi_4 \\
    \pi_1 & \pi_2 & -(\pi_1+\pi_2) & 0 \\
    \pi_1 & \pi_2  & 0 & -(\pi_1+\pi_2) 
\end{pmatrix}\\
&= \kappa A + \lambda B
\end{align*}
Then, checking the square of the basis matrix $A$,
$$A^2=\begin{pmatrix}
    \pi_2(\pi_1+\pi_2) & -\pi_2(\pi_1+\pi_2) & 0 & 0 \\
    -\pi_1(\pi_1+\pi_2) &  \pi_1(\pi_1+\pi_2) & 0 & 0 \\
    0 & 0& \pi_4(\pi_3+\pi_4) & -\pi_4(\pi_3+\pi_4) \\
    0 & 0  & -\pi_3(\pi_3+\pi_4) & \pi_3(\pi_3+\pi_4)
\end{pmatrix}.$$
To have $A^2 \in \text{HKY}_{\pi}$, we require $A^2=\alpha A$ for some $\alpha \in \RR$, as any dependence on $B$, would force the upper right and lower left blocks to be non-zero. However, such an $\alpha$ exists if and only if $\pi_1+\pi_2=0, \ \pi_3+\pi_4=0$, or $\pi_1+\pi_2 = \pi_3+\pi_4$, and as we are taking $\pi$ to be generic this is not necessarily the case. Hence $A^2\notin \text{HKY}_{\pi}$ and so $\text{HKY}_{\pi}$ does not form a Jordan algebra. Therefore, by Theorem \ref{thm:jordanIffExponential} and Lemma \ref{lem:TRUnififfTRpiUnif}, $\text{HKY}^+$ is not uniformization stable.

Taking $n\!=\!4$, we can explicitly list all $G$-equivariant time-reversible models by following the same general procedure outlined in this section for the TN and HKY models for each $\text{TR}_{G}^+$ for arbitrary subgroups $G \leq S_4$. That is, we construct a basis for the linear space $\text{TR}_{(\pi,G)}$, and check whether the Jordan products of these basis matrices are elements of $\text{TR}_{(\pi,G)}$. 
All possible inequivalent $G$-equivariant time-reversible models for $n\!=\!4$ are presented in Table 1. 
By inequivalent, we mean that no two models in the table have rate matrices that are permutation similar to the rate matrices of another model.

Many of the models in Table 1 are well known time-reversible models in the phylogenetic modelling literature \cite[Chap 13]{felsenstein}. The model K81u (K3STu) is the Kimura 3 parameter model but with unequal base frequencies (generic $\pi$) \cite{Kimura_1981}. The `transition models' TIM and TIM3 are captured by the two distinct $S_4$ subgroups of order 2 generated by $(12)$ and $(12)(34)$ respectively. The TIM3 model is uniformization stable, however the TIM model is not. Further, a quick check shows that the other transition model TIM2 is equivalent to TIM3 (take the copy of $S_2<S_4$ generated by the transposition $(34)$) and is hence also uniformization stable. The names of these models are consistent with the well known phylogenetic inference software IQ-TREE \cite{minh}. The names for the models $M_{12}$ and $M_{24}$ are consistent with the naming convention in the list of models derived in \cite{huelsenbeck}.
\begin{table}
\centering
\begin{tabular}{ |p{2cm}||p{6cm}|p{3cm}|p{3cm}|  }

 \hline
 Subgroup & Generator & Model Name & Uniformization Stable\\
 \hline
 Trivial & $\left<e\right>$ & GTR & Yes\\[4pt]
 $S_2$ & $\left<(12)\right>$ & TIM3 & Yes\\[4pt]
 $S_2$ & $\left<(12)(34)\right>$ & TIM & No\\[4pt]
 $C_4$ & $\left<(1234)\right>$ & $\text{M}_{12}$ & No\\[4pt]
 $V_4$ & $\left<(12)(34),(13)(24),(14)(23)\right>$ & K81u (K3STu) & No\\[4pt]
 $V_4$ & $\left<(12),(34),(12)(34)\right>$ & TN93 & Yes\\[4pt]
 $D_4$ & $\left<(1324),(12)\right>$ & HKY & No\\[4pt]
 $A_3$ & $\left<(132),(123)\right>$ & $\text{M}_{24}$ & Yes\\[4pt]
 $S_3$ & $\left<(12),(123)\right>$ & $\text{M}_{24}$ & Yes\\[4pt]
 $A_4$ & $\left<(123),(12)(34)\right>$ & F81 (EI) & Yes\\[4pt]
 $S_4$ & $\left<(1234),(12)\right>$ & F81 (EI) & Yes \\
 \hline
\end{tabular}
\caption{The $G$-equivariant time-reversible models for $n=4$}
\end{table}

Although the equivariant Markov models (as defined in Example \ref{exam:jordanUniformization}) are uniformization stable, we have shown that the models captured by extending this idea to time-reversible models are not uniformization stable in general. It remains a topic for future work to describe all uniformization stable $G$-equivariant time-reversible models for arbitrary $n$.
In the next section we give a further characterisation of uniformization stable Markov models beyond the time-reversible case considered here.
\section{\texorpdfstring{Jordan-Markov models with $S_n$ symmetry}{Jordan-Markov models with Sn symmetry}}
\label{sec:symmetry}

Without further constraints, such as the time-reversibility conditions discussed in the previous section, generating a complete (finite or even discrete) list of all Jordan-Markov models is infeasible. 
However, as was done for the Lie-Markov models in \cite{sumnerLie}, a natural family of constraints comes from enforcing a symmetry by demanding invariance under a group of state permutations.
In this section, we make this notion precise, review the relevant aspects of representation theory we require, and conclude with a characterisation of Jordan-Markov models with `full' $S_n$ permutation symmetry.

Before proceeding further we note that this section relies heavily on results from representation theory, particularly the representation theory of the symmetric group, which is well introduced in \cite{sagan}.

\begin{definition}
\label{def:groupAction}
Given a permutation subgroup $G \leq S_n$, recall that $\text{Mat}_n(\mathbb{R})$ carries an action of $G$ given by simultaneous row and column permutations. 
That is, for all $\sigma\in G$ and $X\in \text{Mat}_n(\mathbb{R})$, we take,
\[
\sigma \cdot X := K_{\sigma}^T X K_{\sigma},
\]
where $K_{\sigma}$ is the standard $n \times n$ permutation matrix associated to the permutation $\sigma$, so $K_{\sigma}^T=K_{\sigma}^{-1}=K_{\sigma^{-1}}$. 
We refer to this as the \emph{conjugation action} of $S_n$ on $\text{Mat}_n(\mathbb{R})$.

More generally, suppose $U\subseteq \text{Mat}_n(\mathbb{R})$ is a linear subspace of matrices, then $U$ forms a $G$-\emph{module} if $U$ is closed under the action of $G$ induced from $\text{Mat}_n(\mathbb{R})$. 
Finally, $U$ is \emph{irreducible} if $U$ does not itself contain any non-trivial $G$-submodules.
\end{definition}
For instance, $\zerocolsum_n$ forms an $S_n$-module since the zero row sums condition is preserved under the action of $S_n$.
This motivates:
\begin{definition}
\label{def:modelSnSymm}
We say a linear Markov model $\zerocolsum^+ := \zerocolsum \cap \zerocolsum_n^+$, with $\zerocolsum$ minimal, has \emph{full permutation symmetry}, or simply $S_n$\emph{-symmetry}, if $\zerocolsum$ is invariant under the action of $S_n$ given in Definition~\ref{def:groupAction}; that is, if $\zerocolsum$ forms an $S_n$-module.
\end{definition}
For example, the general time-reversible model discussed in the previous section has full permutation symmetry, whereas the TN model does not.

The results in the remainder of this chapter rely on the well-known Maschke's theorem. We first recall that the $G$-isomorphic classes of irreducible G-submodules are in bijective correspondence with the conjugacy classes of $G$. In particular, the irreducible submodules of $S_n$ are labelled by the integer partitions of $n$. 
Additionally,
\begin{theorem} (Maschke's Theorem)
\label{thm:maschkesTheorem}
Suppose $G$ is a finite group and $V$ is a $G$-module.
There exist irreducible $G$-submodules $W^{(i)}$ of $V$ such that,
$$V=W^{(1)} \oplus W^{(2)} \oplus ... \oplus W^{(k)} $$
where $k\in \mathbb{N}$ and $k\leq \dim(V)$.

\end{theorem}
\begin{proof}
See \cite[Thm 1.5.3]{sagan}.
\end{proof}

For convenience, up to isomorphism we notate the irreducible modules of $S_n$ by using curly brackets containing the relevant integer partition.
In the present context, the simplest example follows by taking the matrix $J$, hereby defined as the rate matrix with off-diagonal entries $\frac{1}{n}$, and observing $K_{\sigma}^T JK_{\sigma}=J$. Hence $\text{span}_{\RR}(J)$ forms the trivial $S_n$-module, which is naturally labelled by the trivial integer partition:
\[
\text{span}_{\RR}(J)\cong \{n\}.
\]

For our present purposes, we find:
\begin{theorem}
\label{thm:Lndim}
For $n>3$, 
\begin{align*}
\zerocolsum_n&\cong \left\{n\right\}\oplus 2\left\{n-1,1\right\}\oplus \left\{n-2,2\right\}\oplus \left\{n-2,1^2\right\},
 \end{align*}
together with
$
 \zerocolsum_3 \cong \left\{3\right\}\oplus 2\left\{2,1\right\}\oplus  \left\{1^3\right\}$
 and
 $
 \zerocolsum_2\cong \left\{2\right\}\oplus \left\{1^2\right\}.
$
\end{theorem}
\begin{proof}
We begin by noting that the following argument is complete but does require some specialist knowledge of the representation theory of the symmetric group.

After presently establishing the general $n\!>\!3$ case, the $n\!\leq 3\!$ cases are easily verified as an exercise (or by simply noting the invalid integer partitions are removed).

Observing that $S_n$ acts separately on the off-diagonal and diagonal elements of matrices in $\mat_n(\RR)$. 
That is, if $X=(x_{ij})\in \mat_n(\RR)$ and $\sigma\in S_n$, we have $K^T_\sigma X K_\sigma=(x_{\sigma^{-1}(i)\sigma^{-1}(j)})$ so the off-diagonal elements of $X$ map as $x_{ij}\mapsto x_{\sigma^{-1}(i)\sigma^{-1}(j)}$
 with $i\neq j\implies \sigma^{-1}(i)\neq \sigma^{-1}(j)$, and the diagonal elements map as $x_{ii}\mapsto x_{\sigma^{-1}(i)\sigma^{-1}(i)}$ with $\sigma^{-1}(i)=\sigma^{-1}(i)$.
 This means the diagonal entries of matrices in $\mat_n(\RR)$ constitute a copy of the usual $S_n$ action on $\RR^n$ under the coordinate vector identification $x_i\equiv x_{ii}$.
 Additionally, the action of $S_n$ on the elementary rate matrices $L_{ij}$ with $i\neq j$ (Definition~\ref{def:elementaryRateMatrix}) is given by $L_{ij}\mapsto L_{\sigma(i)\sigma(j)}=K^T_\sigma L_{ij}K_\sigma$ and it is hence clear that this is isomorphic to the action of $S_n$ on the off-diagonal entries of matrices in $\mat_n(\RR)$.
 We conclude that we have the following isomorphism of $S_n$ modules: $\mat_n(\RR)\cong\zerocolsum_n\oplus \RR^n$. 

We now note that as an $S_n$-module we clearly have $\mat_n(\RR)\cong \RR^n\otimes \RR^n$, and we recall the well-known decomposition into irreducible modules $\RR^n\cong\{n\}\oplus \{n-1,1\}$, where $\{n\}$ is spanned by the single vector $e_1+e_2+\ldots+e_n$ and $\{n-1,1\}$ is spanned by the vectors $e_1-e_2,e_1-e_2,\ldots , e_1-e_n$ (see \cite{sagan} Examples 2.36 and 2.38).
Since $\{n\}$ corresponds to the trivial representation of $S_n$, it hence behaves as the identity under tensor products and we can apply distributivity to obtain,
\[
\mat_n(\RR)\cong \left(\{n\}\oplus  \{n-1,1\}\right)\otimes \left(\{n\}\oplus  \{n-1,1\}\right)=\{n\}\oplus 2\{n-1,1\} \oplus \left(\{n-1,1\}\otimes \{n-1,1\}\right).
\]
Now suppose $\sigma\in S_n$ has $r_k$ cycles of length $k$, an application of the Frobenius character (trace) formula (see \cite{sagan} Theorem~4.22) yields the formulae,
\begin{align*}
   \chi_{\{n\}}(\sigma)=1,\quad \chi_{\{n-1,1\}}(\sigma)=r_1\!-\!1,\quad
   \chi_{\{n-2,2\}}(\sigma)=\fra{1}{2}r_1(r_1\!-\!3)\!+\!r_2,\quad
   \chi_{\{n-2,1^2\}}(\sigma)=\fra{1}{2}(r_1\!-\!1)(r_1\!-\!2)\!-\!r_2.
\end{align*}
Comparing these to 
\[
\chi_{\{n-1,1\}\otimes \{n-1,1\}}(\sigma)=\chi_{\{n-1,1\}}(\sigma)^2=(r_1-1)^2,
\]
the uniqueness of the decomposition into irreducible components leads to
\[
\{n-1,1\}\otimes \{n-1,1\} = \{n\}\oplus \{n-1,1\}\oplus \{n-2,2\}\oplus \{n-2,1^2\}.
\]
Comparing to $\mat_n(\RR)\cong \zerocolsum_n\oplus \RR^n$ yields the stated result.

\end{proof}

We corroborate this result by noting the so-called `hook length' formula \cite[Chap 3.10]{sagan} yields the dimensions $\dim\left(\{n\}\right)=1$, $\dim\left(\{n-1,1\}\right)=n-1$, $\dim\left(\{n-2,2\}\right)=\fra{1}{2}n(n-3)$, $\dim\left(\{n-2,1^2\}\right) = \fra{1}{2}(n-1)(n-2)$ and hence:
\[
n(n-1)=\dim(\zerocolsum_n)=1+2(n-1)+\fra{1}{2}n(n-3)+
\fra{1}{2}(n-1)(n-2),
\]
as required.

The approach we now follow mimics what was presented in \cite{sumnerLie} and is motivated by the following observation: if $\zerocolsum\subseteq \zerocolsum_n$ is an $S_n$-module
then, following Theorem~\ref{thm:Lndim}, there must exist $a_1,a_3,a_4\in \{0,1\}$ and $a_2\in\{0,1,2\}$ such that
\[
\zerocolsum=a_1\left\{n\right\}\oplus a_2\left\{n-1,1\right\}\oplus a_3\left\{n-2,2\right\}\oplus a_4\left\{n-2,1^2\right\}.
\]
This puts a very strong constraint on the Jordan-Markov models with $S_n$ symmetry and enables us to identify all the possibilities for general $n$.

In fact, we can already say a little more at this point. The following result tells us that any Markov model with full $S_n$ permutation symmetry is either trivial (that is, contains no non-zero rate matrices) or includes the rate matrix $J$.

\begin{lemma}
\label{lem:trivialMarkovModel}
Suppose $ \zerocolsum\subseteq \zerocolsum_n$ is a non-zero $S_n$-module. 
Then $\zerocolsum^+$ is non-trivial in the sense that there exists $ Q\in \zerocolsum^+$ such that $Q\neq 0$ if and only if $\{n\}\cong\text{span}_{\RR}\left(J\right)\subseteq \zerocolsum$.

\end{lemma}
\begin{proof}
If $\text{span}_{\RR}\left(J\right)\subseteq \zerocolsum$, then $0\neq J \in \zerocolsum^+$. 
Conversely, taking $0\neq Q\in \zerocolsum^+$, consider
\[
\widehat{Q}:=\sum_{\sigma\in S_n}K_\sigma^T QK_{\sigma}\in \zerocolsum.
\]
Now, as each $0\neq K_\sigma^T QK_{\sigma}\in \zerocolsum^+$, it follows that $0\neq \widehat{Q}\in \zerocolsum^+$.
Observing that $K_\sigma^T \widehat{Q}K_{\sigma}=\widehat{Q}$ for each $\sigma\in S_n$ and applying the uniqueness of the trivial $S_n$-module $\{n\}$ in the decomposition of $\mathcal{L}_n$ (Theorem~\ref{thm:Lndim}) tells us that $\widehat{Q}=\lambda J$ for some $\lambda >0$ and completes the proof.
\end{proof}

For historical reasons the model spanned by $J$ is sometimes referred to as the $n$-state `Jukes-Cantor' model \cite{jukes}. Here we will refer to it as the constant input model, denoted $\text{CI}_n$, which is consistent with \cite{sumnerEmbed}. To further explore the possibilities, we continue by considering the `equal-input' model, discussed in \cite[Chap 7.3.1]{steel}, as an $S_n$-symmetric submodel of $\zerocolsum_n$:
\[
\text{EI}_n=\spam_\RR\left(R_i:i\in [n]\right),
\]
where $R_i\in \zerocolsum_n^+$ is the rate matrix with $1$'s on the off-diagonal entries of the $i^{th}$ column and 0 on every other off-diagonal entry.
An easy calculation shows that $K_{\sigma}^T R_iK_{\sigma}=R_{\sigma(i)}$ and hence, as $S_n$-modules,
\[
\text{EI}_n\cong\mathbb{R}^n\cong \left\{n\right\}\oplus \left\{n-1,1\right\}.
\]

In what is to follow, the two-fold multiplicity of $\{n-1,1\}$ appearing in the decomposition of $\zerocolsum_n$ creates some complications, so, similarly defining the matrix $C_i$ as the rate matrix with $1$'s on the off-diagonal of the $i^{th}$ row and $0$'s on the remaining off-diagonal entries, we note here that $K_{\sigma}^T C_iK_{\sigma}=C_{\sigma(i)}$ so,
for all (fixed) choices $\mu,\nu\in \mathbb{R}$ not both equal to zero:
\[
\spam_\RR\left(\mu R_i+\nu C_i:i\in [n]\right)\cong \{n\}\oplus \{n-1,1\}.
\]
Additionally, $nJ=\sum_{i\in [n]}R_i=\sum_{i\in [n]}C_i$ tells us that
\[
\spam_\RR\left(R_i,C_j:i,j\in [n]\right)\cong \{n\}\oplus 2\{n-1,1\}.
\]

Further, we recall that a `doubly stochastic' rate matrix has both column and row sums equal to 0 and the Birkhoff–von Neumann Theorem (see \cite[Thm 2.1.6]{bapat}), which states that any doubly stochastic rate matrix is a conical combination of the matrices $L_\sigma:=K_\sigma-I_n$.
This motivates
\[
\text{DS}_n:=\left\{Q\in \zerocolsum_n:Q\oneVec =0=\oneVec ^T Q\right\}=
\spam_\RR\left(L_\sigma:\sigma\in S_n\right),
\]
which is an $S_n$-module with $\dim(\text{DS}_n)=(n-1)^2$.
A short calculation shows that $\text{EI}_n\cap \text{DS}_n=\spam_\RR\left(J\right)\cong\{n\}$ and $\mu R_i+\nu C_i\in \text{DS}_n$ if and only if $\mu=\nu$. 
Thus, comparing to Theorem~\ref{thm:Lndim}, we see that: 
\[
\text{DS}_n\cong \left\{n\right\}\oplus \left\{n-1,1\right\}\oplus \left\{n-2,2\right\}\oplus \left\{n-2,1^2\right\}.
\]

We observe that the transpose of a rate matrix is a rate matrix if and only if it is doubly stochastic, and, since the transpose operation commutes with the $S_n$ action, we can decompose $\text{DS}_n$ into symmetric and anti-symmetric parts: $\text{DS}_n=\text{Symm}_n\oplus \text{Anti}_n$.
Since $\dim(\text{DS}_n)=(n-1)^2$, it follows that $\dim(\text{Anti}_n)=\binom{n-1}{2}$, so comparing to the dimension formulae above yields:
\[
\text{Symm}_n=\spam_\RR\left(L_\sigma+L_{\sigma^{-1}}:\sigma\in S_n\right)\cong \left\{n\right\}\oplus \left\{n-1,1\right\}\oplus \left\{n-2,2\right\}
\]
and
\[
\text{Anti}_n=\spam_\RR\left(L_\sigma-L_{\sigma^{-1}}:\sigma\in S_n\right)\cong \left\{n-2,1^2\right\},
\]
with, in particular,
\[
\text{Symm}_n\supseteq\spam_\RR\left(R_i+C_i:i\in [n]\right)\cong \{n\}\oplus \{n-1,1\}.
\]

We record the following algebraic properties of the submodules identified thus far.
\begin{lemma}
\label{lem:submoduleslist}
\begin{enumerate}
\item[]
    \item $\text{EI}_n\cong \{n\}\oplus\{n-1,1\}$ is a matrix algebra.
    \item $\text{DS}_n\cong \{n\}\oplus\{n-1,1\}\oplus\{n-2,2\}\oplus\{n-2,1^2\}$ is a matrix algebra.
    \item $\text{Symm}_n\cong \{n\}\oplus\{n-1,1\}\oplus\{n-2,2\}$ is a Jordan algebra, but not a Lie algebra.
     \item $\text{Anti}_n\cong \{n-2,1^2\}$ is a Lie algebra, but not a Jordan algebra.
    \item The subspace sum $\text{EI}_n+\text{Symm}_n\cong \{n\}\oplus 2\{n-1,1\}\oplus\{n-2,2\}$ is a Jordan algebra but not a Lie algebra.
   
\end{enumerate}
\end{lemma}
\begin{proof}
\begin{enumerate}
\item[]
    \item The result follows from $R_iR_j=-R_i$ for all $i,j\in [n]$.
    \item Both zero row and zero column sums are preserved under matrix multiplication.
    \item The Jordan product of two symmetric matrices is again symmetric, whereas the Lie product of two symmetric matrices is anti-symmetric.
    \item The Lie product of two anti-symmetric matrices is again anti-symmetric, whereas the Jordan product of two anti-symmetric matrices is symmetric.
    \item Noting $\text{EI}_n+\text{Symm}_n=\spam_\RR\left(R_i,L_\sigma+L_{\sigma^{-1}}:i\in [n],\sigma\in S_n\right)$, it is straightforward to establish the result using (1) and (3) and by confirming the relations:
    \[
    L_{\sigma}R_i=-L_\sigma, \qquad R_iL_{\sigma}=R_{\sigma(i)}-R_i-L_\sigma.
    \]
\end{enumerate}
\end{proof}

Since $S_n$ acts by conjugation, we have:
\begin{fact}
\label{fact:typesubspace}
For each cycle type $\lambda$, the following are $S_n$-modules:
\[
\spam_\RR\left(L_\sigma+L_{\sigma^{-1}}:\sigma \in S_n,\text{ cycle type }\lambda\right)\subseteq \text{Symm}_n,
\] 
which is possibly reducible, and 
\[
\spam_\RR\left(L_\sigma-L_{\sigma^{-1}}:\sigma \in S_n,\text{ cycle type }\lambda\right)\subseteq \text{Anti}_n,
\]
which, by the irreducibility of $\text{Anti}_n$, is either zero or equality holds.
\end{fact}

At this point, with the exception of the $\{n-2,2\}$ case, we have explicitly identified matrices that span each of the submodules present in the decomposition of $\mathcal{L}_n$ (Theorem~\ref{thm:Lndim}). 
To rectify this situation, we define, for all $Q\in \zerocolsum_n$,
\[
Q_z:=Q\circ (\oneVec \oneVec^T -I_n),
\]
where `$\circ$' indicates the entrywise matrix product, so $Q_z$ is the obtained from $Q$ by setting all entries on the diagonal to zero.
Then:
\begin{lemma}
The sum of squares of off-diagonal entries, $\text{tr}(Q_z^TQ_z)$, gives an $S_n$-invariant inner product on $\zerocolsum_n$, with $\langle Q,Q' \rangle:=\text{tr}(Q_z^TQ'_z)$ for all $Q,Q'\in \zerocolsum_n$.
\end{lemma}
\begin{proof}
That we have an inner product is clear.
Since the conjugation action of $S_n$ on $\mat(n,\RR)$ acts independently on the diagonal and off-diagonal entries, it follows that $(K^{-1}QK)_z=K^{-1}Q_zK$ for all permutation matrices $K$.
The result then follows from $K^{-1}=K^T$ and the cyclic property of the trace.
\end{proof}

\begin{fact}
Under a group invariant inner product, orthogonal complements of submodules are themselves submodules.
\end{fact}
\begin{proof}
See \cite[Prop 1.5.2]{sagan}.
\end{proof}
Thus to explicitly identify matrices $X\in \zerocolsum_n$ which lie in the $S_n$-submodule isomorphic to the $\{n-2,2\}$, we look for symmetric zero row sum matrices $X$ that satisfy $\langle R_i+C_i,X\rangle=0$ for each $i$.
Thus: 
\begin{lemma}
\label{lem:zeroColSumSymmSpace}
The submodule of zero row sum matrices in $\mathcal{L}_n$ isomorphic to $\{n-2,2\}$ is the subspace of symmetric zero row sum matrices with zero entries on the diagonal.
That is,
\[
\left\{Q=(q_{ij})\in \zerocolsum_n:Q=Q^T,q_{ii}=0,\forall i\in [n]\right\}\cong \{n-2,2\}.
\]
This submodule can be expressed as
\[
\spam\left(L_{\sigma_1}-L_{\sigma_2}:e=\sigma_1^2=\sigma_2^2,\fix(\sigma_1)=\fix(\sigma_2)\right)
\]
or alternatively, for each cycle type $\lambda=(r_1,r_2,\ldots,r_s)$ with $1\leq r_i\leq 2$:
\[
\spam\left(L_{\sigma_1}-L_{\sigma_2}:\fix(\sigma_1)=\fix(\sigma_2),\lambda(\sigma_1)=\lambda(\sigma_2)=(r_1,r_2,\ldots,r_s)\right),
\]
so, in particular, choosing the cycle type $(2,2,1,1,\ldots,1)$, we have:
\[
\spam\left(L_{(ij)(kl)}-L_{(ik)(jl)}:|\{i,j,k,l\}|=4\right)\cong\left\{n-2,2\right\}.
\]
\end{lemma}
\begin{proof}
The first characterisation follows simply from $Q^T=Q$ and $\langle R_i+C_i,Q\rangle=0$, for each $i\in [n]$. 
The further characterisations follow from Fact~\ref{fact:typesubspace} and observing that a sum of permutation matrices gives a symmetric zero row sum matrix with zero on the diagonal if and only if the stated conditions are met.

\end{proof}

To assist with the double multiplicity of the irreducible module $\{n-1,1\}$ in the decomposition of $\zerocolsum_n$ (Theorem~\ref{thm:Lndim}), at this point it is helpful to explicitly record the following Jordan products.
\begin{lemma}
\label{lem:prods1}

\begin{enumerate}
\item[]
\item $R_i\odot R_j=-(R_i+R_j)$.
\item $R_i\odot C_j=n\delta_{ij}\left(J-R_j\right)-2C_j$.
\item $C_i\odot C_i=-2(n-1)C_i$.
\item $C_i\odot C_j=C_i+C_j-n(L_{ij}+L_{ji})$, for $i\neq j$.
\end{enumerate}

\end{lemma}
\begin{proof}
We obtain (1) from $R_iR_j=-R_i$.
The other results are obtained using the convenient forms $R_i=T_i-I_n$ and $C_i=U_i-nP_{i}$ where $T_i$ is the matrix with $1$'s on the $i^{th}$ column and zeroes elsewhere, $U_i$ is the matrix with $1$'s in the $i^{th}$ row and zeroes elsewhere and $P_i$ is the matrix with only nonzero entry 1 appearing in $i^{th}$ diagonal position. We then have,
$$R_iC_j = (T_i-I_n)(U_j-nP_j) = T_iU_j -nT_iP_j -U_j+nP_j=\delta_{ij}n(J+I_n)-n\delta_{ij}T_i-U_j+nP_j.$$
Hence, if $i \neq j$, $R_iC_j=-C_j$ and $R_iC_i = n(J-R_i)-C_j$. 
Similarly, $C_jR_i=U_jT_i-U_j-nP_jT_i+nP_j=-C_j \ \forall \ i,j$, which yields (2).

Considering products of the $C_i$ matrices we obtain
$$\begin{cases}
C_jC_i =U_j-nL_{ij}-nP_j = C_j-nL_{ij}, & i \neq j, \\
C_jC_i = U_i-nU_i-nP_i+n^2P_i = -(n-1)C_i, & i = j; \\
\end{cases}$$
and
$$\begin{cases}
C_iC_j = C_i-nL_{ji}, & i \neq j, \\
C_iC_j = -(n-1)C_i, & i = j; \\
\end{cases}$$
which yields (3) and (4).

\end{proof}

We now present the three lemmas that provide the means for proving the classification presented below in Theorem~\ref{thm:Snclassification}.

\begin{lemma}
\label{lem:symmEILemma}
Suppose $n>2$ and $ \zerocolsum\subseteq \zerocolsum_n$ is a Jordan algebra and $S_n$-module containing both $\spam_\RR\left(J\right)\cong \{n\}$ and at least one $S_n$-submodule isomorphic to $\{n-1,1\}$.
Then precisely one of the following possibilities hold: 

\begin{enumerate}
    \item $\text{EI}_n\subseteq \zerocolsum$ and $\text{Symm}_n\nsubseteq \zerocolsum$;
    \item $\text{Symm}_n\subseteq \zerocolsum$ and $\text{EI}_n\nsubseteq \zerocolsum$;
    \item $\text{EI}_n+\text{Symm}_n\subseteq \zerocolsum$.
\end{enumerate} 
\end{lemma}
\begin{proof}
Throughout, assume $\zerocolsum$ satisfies the stated conditions.
\begin{itemize}
\item 
Suppose $\zerocolsum$ contains \emph{precisely one} submodule isomorphic to $\{n-1,1\}$. 

Since $R_i+C_i\in \text{Symm}_n$, we see that 
\[
\{n\}\oplus 2\{n-1,1\}\cong\spam_\RR(R_i,C_j:i,j\in [n])=\spam_\RR(R_i,R_j+C_j:i,j\in [n])\subset \text{EI}_n+\text{Symm}_n,
\]
and we hence cannot have both $\text{EI}_n\subseteq \zerocolsum$ and $\text{Symm}_n\subseteq \zerocolsum$.
Additionally, $\{n\}\cong \spam_\RR(J)\subset \zerocolsum$ implies there must exist fixed $\mu,\nu\in \mathbb{R}$, not both equal to zero, such that $\spam_{\RR}\left(\mu R_i+\nu C_i\right)\subseteq \zerocolsum$.

If $\nu=0$ and $\mu\neq 0$, we have $\spam_\RR(R_i:i\in [n])=\text{EI}_n\subseteq \zerocolsum$, and case (1) is established.

If $\mu=0$ and $\nu\neq 0$, taking $i\neq j$ we have $C_i,C_j\in \zerocolsum$ and, recalling Lemma~\ref{lem:prods1}:
\begin{align*}
C_i\odot C_j=C_i+C_j-n\left(L_{ij}+L_{ji}\right)\in \zerocolsum.
\end{align*}
This shows that each $L_{ij}+L_{ji}\in \zerocolsum$ and hence $\spam_\RR(L_{ij}+L_{ji}:i,j\in [n], i\neq j)=\text{Symm}_n\subseteq \zerocolsum$. 
However, $R_i+C_i\in \text{Symm}_n$ gives $R_i=(R_i+C_i)-C_i\in \zerocolsum \cap \text{EI}_n$, which, from our initial observation above, contradicts the assumptions on $\mathcal{L}$.

If both $\mu,\nu\neq 0$,
using Lemma~\ref{lem:prods1}, we calculate:
\[
\left(\mu R_i+\nu C_i\right)\odot \left(\mu R_i+\nu C_i\right)=-2\mu^2R_i+2\mu \nu\left[n(J-R_i)-2C_i\right]-2\nu^2(n-1) C_i\in \zerocolsum.
\]
However, by assumption $J\in \zerocolsum$ and hence
\[
-2\mu^2R_i+2\mu \nu\left[-nR_i-2C_i\right]-2\nu^2(n-1) C_i\in \zerocolsum,
\]
also.
Linear independence implies there must exist $\lambda\neq 0$ such that
\[
-2\mu^2R_i+2\mu \nu\left[-nR_i-2C_i\right]-2\nu^2(n-1) C_i=\lambda\left(\mu R_i+\nu C_i\right).
\]
From this we are led to $\mu=\nu$ and hence $R_i+C_i\in \zerocolsum$.
Considering
    \[
    \left( R_i+C_i\right)\odot \left(R_j+ C_j\right)=-\left( R_i+C_i\right)-\left( R_j+C_j\right)-n\left(L_{ij}+L_{ji}\right)\in\zerocolsum,
    \]
 shows that each $L_{ij}+L_{ji}\in \zerocolsum$ also and, since these elements span $\text{Symm}_n$, we see that $\text{Symm}_n\subseteq \zerocolsum$, which establishes case (2).

\item

On the other hand, if $\{n\}\oplus 2\{n-1,1\}\cong \spam_\RR\left(R_i,C_j:i,j\in [n]\right)\subseteq \zerocolsum$, we have $\spam_\RR\left(R_i:i\in [n]\right)=\text{EI}_n\subset \zerocolsum$ and, referring above, $C_i\odot C_j\in \zerocolsum$ implies $\text{Symm}_n\subseteq \zerocolsum$, which establishes case (3).

\end{itemize}

\end{proof}

The following lemma is only valid for $n>4$. 
There is nothing to say for $n<4$, and we treat the case $n=4$ separately below.

\begin{lemma}
\label{lem:symmLemma}
Suppose $n>4$ and $ \zerocolsum\subseteq \zerocolsum_n$ is a Jordan algebra and $S_n$-module containing $\spam_\RR(J)\cong \{n\}$ and the submodule isomorphic to $\{n-2,2\}$ (c.f. Lemma~\ref{lem:zeroColSumSymmSpace}).
Then $\{n\}\oplus\{n-1,1\}\oplus\{n-2,2\}\cong \text{Symm}_n\subseteq \zerocolsum$.
\end{lemma}
\begin{proof}
By Lemma \ref{lem:zeroColSumSymmSpace} we have
\[
Q_1=L_{(12)(34)}-L_{(13)(24)},\quad Q_2=L_{(12)(35)}-L_{(13)(25)}\in \zerocolsum.
\]
Direct computation (the $5\times 5$ case is sufficient) shows that $Q_1\odot Q_2$ is a symmetric matrix with non-constant diagonal entries, and hence by Lemma~\ref{lem:zeroColSumSymmSpace} cannot be contained in the submodules of $\zerocolsum$ isomorphic to $\{n\}$ or $\{n-2,2\}$.
Lemma~\ref{lem:submoduleslist} then tells us that $\{n\}\oplus\{n-1,1\}\cong \spam_\mathbb{R}\left(R_i+C_i:i\in [n]\right)\subset\zerocolsum$ and hence $\text{Symm}_n\subseteq \zerocolsum$, as required.
\end{proof}

The following result is valid for $n>3$. 
For $n\!=\!2$ there is nothing to say, and for $n\!=\!3$ we find that the submodule isomorphic to $\{3\}\oplus\{1^3\}$ forms a Jordan algebra (as is confirmed below).

\begin{lemma}
\label{lem:doublyStochasticLemma}
Suppose $n>3$ and $ \zerocolsum\subseteq \zerocolsum_n$ is a Jordan algebra and $S_n$-module containing $\text{Anti}_n\cong \{n-2,1^2\}$.
Then $\text{DS}_n\subseteq \zerocolsum$.
\end{lemma}
\begin{proof}
Fact \ref{fact:typesubspace} shows that taking $\sigma=(ijkl)\in S_n$ we have
\[
Q=L_\sigma-L_{\sigma^3}\in \text{Anti}_n\cong \{n-2,1^2\},
\]
and a short calculation yields,
\[
Q\odot Q=4L_{(ik)(jl)}\in \zerocolsum.
\]
An easy extension of Lemma~\ref{lem:zeroColSumSymmSpace} shows that $\spam_\RR\left(L_{(ik)(jl)}:|\{i,j,k,l\}|=4\right)\cong \{n\}\oplus\{n-2,2\}$, and hence for $n>4$ we may apply Lemma~\ref{lem:symmLemma} to conclude that $\{n\}\oplus\{n-1,1\}\oplus\{n-2,2\}\oplus\{n-2,1^2\} \cong \text{DS}_n\subseteq \zerocolsum$, as required. For $n=4$, it can be checked that the Jordan product of two distinct $4\times 4$ antisymmetric rate matrices is a symmetric matrix with non-constant diagonal entries. For example, consider the following Jordan product of antisymmetric rate matrices, 
$$\begin{pmatrixT} 0 & 0 & -1 & 1 \\ 0 & 0 & 1 & -1 \\ 1 & -1 & 0 & 0 \\ -1 & 1 & 0 & 0  \end{pmatrixT} \odot \begin{pmatrixT} 0 & -1 & 0 & 1 \\ 1 & 0 & -1 & 0 \\ 0 & 1 & 0 & -1 \\ -1 & 0 & 1 & 0  \end{pmatrixT} = \begin{pmatrixT} -2 & 0 & 0 & 2 \\ 0 & 2 & -2 & 0 \\ 0 & -2 & 2 & 0 \\ 2 & 0 & 0 & -2 \end{pmatrixT}. $$ 
 Hence $\{4\}\oplus\{3,1\}\cong \spam_\mathbb{R}\left(R_i+C_i:i\in [4]\right)\subset\zerocolsum$ and so $\text{DS}_n \subseteq \zerocolsum$.
\end{proof}
Putting these lemmas together we obtain the following theorem and the hierarchies of Jordan-Markov models presented in Figure~\ref{fig:hierarchies}.
The cases $n\!=\!3$ and $n\!=\!4$ include additional cases due to the appearance of the normal subgroups $C_3$ and $V_4$, respectively, and hence fall under the purview of `group-based' models (see \cite[Chap 7.3.2]{steel}, and also \cite{woodhams} for updated perspectives on this class of models).

\begin{theorem}
\label{thm:Snclassification}
The nontrivial, linear, uniformization stable Markov models with full $S_n$ symmetry occur as intersections with $\zerocolsum_n^+$ and the following linear subspaces:
\begin{enumerate}
    \item[] For $n>4$:
    \begin{itemize}
        \item The `constant input' model: $\text{CI}_n=\spam_{\RR}\left(J\right)$, with $\dim(\text{CI}_n)=1$.
        \item The `equal input' model: $\text{EI}_n=\spam_{\RR}\left(R_i:i\in [n]\right)$, with $\dim(\text{EI}_n)=n$.
        \item The `symmetric' model: $\text{Symm}_n=\spam_{\RR}\left(L_{ij}+L_{ji}:i,j\in [n]\right)$, with $\dim(\text{Symm}_n)=\fra{1}{2}n(n-1)$.
         \item  $\text{EI}_n+\text{Symm}_n$, with $\dim(\text{EI}_n+\text{Symm}_n)=\fra{1}{2}(n+2)(n-1)$.
         \item The `doubly stochastic' model: $\text{DS}_n=\spam_{\RR}\left(L_\sigma:\sigma\in S_n\right)$ with $\dim(\text{DS}_n)=(n-1)^2$.
         \item The `general Markov' model $\text{GM}_n=\spam_{\RR}\left(L_{ij}:1\leq i,j \leq n\right)=\mathcal{L}_n$ with $\dim(\text{GM}_n)=n(n-1)$.
    \end{itemize}
    \item[] For $n\!=\!2$:
    \begin{itemize}
        \item $\text{CI}_2=\spam_\RR\left(J\right)$ with $\dim\left(\text{CI}_2\right)=1$.
        \item $\text{GM}_2=\text{EI}_2=\zerocolsum_2$, with $\dim\left(\zerocolsum_2\right)=2$.
    \end{itemize}
    \item[]  For $n\!=\!3$, each of the cases valid for the general $n>4$ is included together with two additional cases:
    \begin{itemize}
        \item The `group-based' model: $\zerocolsum_{C_3}=\spam_\RR(L_\sigma:\sigma\in C_3)$ obtained from the cyclic group $C_3$, with $\dim(\zerocolsum_{C_3})=2$.
        \item $\zerocolsum_{C_3}+\text{EI}_3=\spam_\RR\left(R_1,R_2,R_3,L_{(123)}-L_{(132)}\right)$, with $\dim\left(\zerocolsum_{C_3}+\text{EI}_3\right)=4$. 
        \end{itemize} 
    \item[] For $n\!=\!4$, each of the cases valid for the general $n>4$ is included together with two additional cases. We note that these two models are well-known models in phylogenetics, commonly denoted by K3ST, and K3ST+F81 respectively \cite{felsenstein1981,Kimura_1981}:
    \begin{itemize}
        \item The `group-based' model: $\zerocolsum_{V_4}=\spam_\RR(L_\sigma:\sigma\in V_4)$ obtained from the Klein 4-group $V_4$, with $\dim\left(\zerocolsum_{V_4}\right)=3$. 
        \item $\zerocolsum_{V_4}+\text{EI}_4=\spam_\RR\left(R_1,R_2,R_3,R_4,L_{\sigma}:\sigma \in V_4\right)$, with $\dim\left(\zerocolsum_{V_4}+\text{EI}_4\right)=6$. 
    \end{itemize}
        
\end{enumerate}

\end{theorem}
\begin{proof}
Throughout, assume $\zerocolsum \subseteq \zerocolsum_n$ is a Jordan algebra and $S_n$-module. 
Recalling Lemma~\ref{lem:trivialMarkovModel}, for all $n\geq 1$ we note that non-triviality of the model implies $\{n\}\cong \spam_\RR\left(J\right)\subseteq \zerocolsum$.
If $\zerocolsum=\text{CI}_n$, we have a Jordan algebra and we are done.
Thus, we may assume $ \text{CI}_n\subset \zerocolsum$.

For $n>4$, taken together Lemmas~\ref{lem:symmLemma} and \ref{lem:doublyStochasticLemma} imply that $\zerocolsum$ must contain a submodule isomorphic to $\{n-1,1\}$, and, in particular, at least one of $\text{Symm}_n$ or $\text{EI}_n$ is contained in $\zerocolsum$. 
Then, if $\zerocolsum$ contains either $\text{Symm}_n$ or $\text{EI}_n$ but is not equal to either, we have $\zerocolsum = \text{EI}_n+\text{Symm}_n$ or $\zerocolsum$ contains a submodule isomorphic to $\{n-2,1^2\}$ and hence $\text{DS}_n \subseteq \zerocolsum$. 
If $\text{DS}_n \subset \zerocolsum$, it must be that $\{n-1,1\} \cong \text{EI}_n \subset \zerocolsum$ also and hence $\zerocolsum = \zerocolsum_n$.

For the $n\!=\!2$ case we have $\zerocolsum_2\cong \{2\}\oplus \{1^2\}$, so, since we are assuming $\zerocolsum\neq \text{CI}_n$, we have $\zerocolsum=\zerocolsum_2\cong \{2\}\oplus\{1^2\}$, which is a Jordan algebra (Lemma~\ref{lem:matrixIffJordanAndLie}).

For the $n\!=\!3$ case, we have $\zerocolsum_3\cong \{3\}\oplus 2\{n-1,1\}\oplus \{1^3\}$.
One easily checks that, for all $\sigma\in S_3$, we have $K{_\sigma}^T\left(L_{(123)}-L_{(132)}\right)K_{\sigma}=\text{sgn}(\sigma)\left(L_{(123)}-L_{(132)}\right)$ and hence $\spam_\RR\left(L_{(123)}-L_{(132)}\right)\cong \{1^3\}$ (the sign representation of $S_3$). 
Thus if $\zerocolsum\cong \{3\}\oplus \{1^3\}$ we have $\zerocolsum=\spam_{\RR}\left(L_\sigma:\sigma\in C_3\right)$, which is a Jordan algebra and also an $S_3$-module (the latter due to the normality of $C_3$ as a subgroup of $S_3$).
If $\zerocolsum\cong \{3\}\oplus \{2,1\}\oplus \{1^3\}$, Lemma~\ref{lem:symmEILemma} implies that either $\text{Symm}_3\subset \zerocolsum$, and so $\zerocolsum=\text{DS}_3$, or $\text{EI}_3\subset \zerocolsum$, and so $\zerocolsum=\text{EI}_3+\zerocolsum_{C_3}$, which is easily confirmed to be a Jordan algebra.

For the $n\!=\!4$ case, we have $\zerocolsum_4 \cong \{4\} \oplus \{3,1\} \oplus \{2^2\} \oplus \{2,1^2\}$. From Lemmas \ref{lem:symmEILemma} and \ref{lem:doublyStochasticLemma}, $\zerocolsum$ must contain a submodule isomorphic to either $\{3,1\}$ or $\{2^2\}$. If $\zerocolsum$ contains a submodule isomorphic to $\{3,1\}$, but not $\{2^2\}$, then the argument follows as for the $n>4$ case. 
If $\zerocolsum \cong \{4\} \oplus \{2^2\}$, then by Lemma \ref{lem:zeroColSumSymmSpace}, $\zerocolsum = \spam_{\RR}(J,L_{(12)(34)}-L_{(13)(24)},L_{(13)(24)}-L_{(14)(23)})$, which is equal to $\zerocolsum_{V_4}=\spam_{\RR}(L_{\sigma}:\sigma \in V_4)$. 
An easy check confirms this is a matrix algebra and, hence, a Jordan algebra. If $\zerocolsum$ contains $\zerocolsum_{V_4}$ and a submodule isomorphic to $\{3,1\}$, by Lemma \ref{lem:symmEILemma}, $\zerocolsum$ either contains $\text{Symm}_4$ or is equal to $\text{EI}_4+\zerocolsum_{V_4}$, which can be confirmed to form a Jordan algebra. From here the proof follows the same as for $n>4$.

\end{proof}

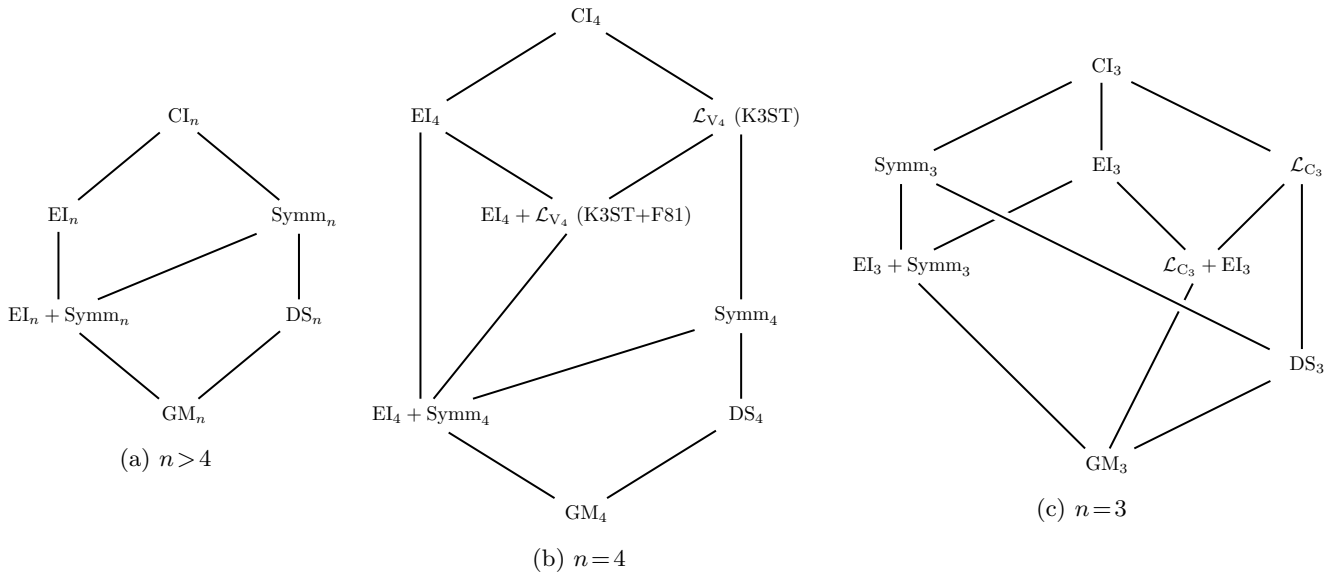
\begin{figure}
\hspace{-3em}
\begin{subfigure}{.33\textwidth}
  \adjustbox{scale=.75,center}{%
  \begin{tikzcd}[arrows={line width=0.1em}, column sep={6em,between origins}, row sep={5em,between origins}]
& \phantom{_n}\mathrm{CI}_n & \\
\phantom{_n}\mathrm{EI}_n \arrow[ru, no head] & & \phantom{_n}\mathrm{Symm}_n \arrow[lu, no head] \\
\phantom{_n}\phantom{_n}\mathrm{EI}_n + \mathrm{Symm}_n \arrow[rru, no head] \arrow[u, no head] & & \phantom{_n}\mathrm{DS}_n \arrow[u, no head]\\  & \phantom{_n}\mathrm{GM}_n \arrow[ru, no head] \arrow[lu, no head] &      
\end{tikzcd}
}
  \caption{$n\!>\!4$}
  \label{fig:sub1}
\end{subfigure}%
\hspace{-1em}
\begin{subfigure}{.33\textwidth}
  \adjustbox{scale=.75,center}{%
  \begin{tikzcd}[arrows={line width=0.1em}, column sep={8em,between origins}, row sep={5em,between origins}] & \phantom{_n}\mathrm{CI}_4 \arrow[ld, no head] \arrow[rd, no head] & \\
\phantom{_n}\mathrm{EI}_4 \arrow[ddd, no head] \arrow[rd, no head] & & \phantom{_n}\mathcal{L}_{\mathrm{V}_4} \ \text{(K3ST)} \arrow[ld, no head] \arrow[dd, no head]   \\&\phantom{_n}\mathrm{EI}_4+\mathcal{L}_{\mathrm{V}_4} \ \text{(K3ST+F81)} \arrow[ldd, no head] &  \\&  & \phantom{_n}\mathrm{Symm}_4 \arrow[lld, no head] \arrow[d, no head] \\
\phantom{_n}\phantom{_n}\mathrm{EI}_4 + \mathrm{Symm}_4 \arrow[rd, no head] &  & \phantom{_n}\mathrm{DS}_4 \arrow[ld, no head] \\  & \phantom{_n}\mathrm{GM}_4 &    
\end{tikzcd}
}
  \caption{$n\!=\!4$}
  \label{fig:sub2}
\end{subfigure}
\hspace{2em}
\begin{subfigure}{.33\textwidth}
  \adjustbox{scale=.75,center}{%
  \begin{tikzcd}[arrows={line width=0.1em}, column sep={5em,between origins}, row sep={5em,between origins}]
& & \phantom{_n}\mathrm{CI}_3 \arrow[rrd, no head]  &       & \\
\phantom{_n}\mathrm{Symm}_3 \arrow[rru, no head]  &  & \phantom{_n}\mathrm{EI}_3 \arrow[u, no head] \arrow[rd, no head]      &                                                             & \phantom{_n}\mathcal{L}_{\mathrm{C}_3} \arrow[dd, no head] \arrow[ld, no head] \\
\phantom{_n}\phantom{_n}\mathrm{EI}_3 + \mathrm{Symm}_3 \arrow[u, no head] \arrow[rru, no head] &  &  & \phantom{_n}\mathcal{L}_{\mathrm{C}_3}+\mathrm{EI}_3 \arrow[ldd, no head] & \\
{}  & {} \arrow[l, phantom] & & & \phantom{_n}\mathrm{DS}_3 \arrow[lllluu, no head,crossing over]\\& &\phantom{_n}\mathrm{GM}_3\arrow[rru, no head] \arrow[lluu, no head] & &                
\end{tikzcd}
}
  \caption{$n\!=\!3$}
  \label{fig:sub3}
\end{subfigure}
\caption{Hierarchies of uniformization stable Markov models with $S_n$ symmetry.}
\label{fig:hierarchies}
\end{figure}

\section{Discussion}

In this article we have established a characterisation of uniformization stable continuous-time Markov chains in terms of Jordan algebras associated to the space spanned by the rate matrices associated to the model.
Although, our main result (Theorem~\ref{thm:jordanIffExponential}) is technically limited to the case of linear models (Definition~\ref{def:markovModel}), we expect this characterisation provides significant insight into the general case, as we illustrated for time-reversible models in Section~\ref{sec:timeReversible}. 

Time-reversible Markov models play an important role in phylogenetic inference and, in Section~\ref{sec:timeReversible}, we displayed that the time-reversibility criterion naturally leads to the consideration of Jordan algebraic structure. While we have shown that time-reversibility of a Markov model does not imply the model is uniformization stable, we have presented a range of well-known time-reversible models of DNA evolution that exhibit uniformization stability, a number of which are popular models in phylogenetic modelling. It is still an open question whether there exist linear, uniformization stable, time-reversible Markov models that are not $G$-equivariant for some $G\leq S_n$. Further, the relevant condition on the group $G$ that leads to a $G$-equivariant time-reversible model being uniformization stable is unknown, and a topic for future work is to derive a hierarchy of uniformization stable $G$-equivariant time-reversible models for arbitrary $n$.

In Section \ref{sec:symmetry} we presented a complete hierarchy of linear uniformization stable Markov models with full $S_n$ symmetry for $n\geq 2$. The hierarchy relevant to molecular phylogenetics ($n=4$) presents itself here as a special case, exhibiting two models not present in the general case due to the presence of the Klein 4-group $V_4\leq S_4$. Many well-known models of DNA evolution present themselves in this hierarchy, including the Jukes-Cantor model, the equal-input model and the doubly stochastic model. Many of these models also appear in the hierarchy of Lie-Markov models with $S_4$ symmetry presented in \cite{sumnerLie}, due to such models forming matrix algebras. Unique among the uniformization stable models are the symmetric model $\text{Symm}_4$ and the equal-input + symmetric model $\text{EI}_4+\text{Symm}_4$ that have Jordan algebraic structure, but not Lie algebraic structure, and hence do not appear in the Lie-Markov Models hierarchy.

It is a topic for future work to derive hierarchies of uniformization stable Markov models with full $G$-symmetry for $G < S_n$, with there being particular interest in the dihedral group $D_4 < S_4$. This interest arises as, in the context of phylogenetics, a Markov model possessing $D_4$ symmetry is equivalent to the model respecting the grouping of nucleotides into purines and pyrimidines \cite{sumnerPurine}. The arguments presented in Section \ref{sec:symmetry} are a generalisation of arguments given in \cite{sumnerLie} for the $4 \times 4$ Lie algebra case. In a similar fashion as for uniformization stable Markov models, a derivation of Lie-Markov models with full $S_n$-symmetry for $n \geq 2$ can be done, though it is a topic for a different time.
\\ \\
\textbf{Declaration of interest}
\\

There are no competing interests.
\\\\
\textbf{Data Availability Statement}
\\

Data sharing is not applicable to this article as no datasets were generated or analysed during the study.
\\\\
\textbf{Acknowledgements}
\\

This work was supported by Australian Research Council Discovery Grant DP 180102215. We would like to thank Joshua Stevenson for helpful discussions throughout the course of this work. We also thank the anonymous reviewers, whose insightful comments assisted us in substantially improving this manuscript.
\small
\bibliographystyle{plain}
\bibliography{jordanBib}

\begin{thebibliography}{10}

\bibitem{sumnerEmbed}
M.~Baake and J.~Sumner.
\newblock Notes on {M}arkov embedding.
\newblock {\em Linear Algebra and its Applications}, 594:262--299, 2020.
\newblock doi:
  \href{https://doi.org/10.1016/j.laa.2020.02.016}{10.1016/j.laa.2020.02.016}.

\bibitem{bapat}
R.~B. Bapat and T.~E.~S. Raghavan.
\newblock {\em Nonnegative Matrices and Applications}.
\newblock Cambridge University Press, 1997.

\bibitem{casanellas}
M.~Casanellas and J.~Fern\'andez-S\'anchez.
\newblock Relevant phylogenetic invariants of evolutionary models.
\newblock {\em Journal de Math\' ematiques Pures et Appliqu\'es},
  96(3):207--229, 2011.
\newblock doi:
  \href{https://doi.org/10.1016/j.matpur.2010.11.002}{10.1016/j.matpur.2010.11.002}.

\bibitem{davies}
E.~Davies.
\newblock Embeddable {M}arkov matrices.
\newblock {\em Electronic Journal of Probability}, 15:1474--1486, 2010.
\newblock doi: \href{https://doi.org/10.1214/ejp.v15-733}{10.1214/ejp.v15-733}.

\bibitem{draisma}
J.~Draisma and J.~Kuttler.
\newblock On the ideals of equivariant tree models.
\newblock {\em Mathematische Annalen}, 344(3):619--644, 2009.
\newblock doi:
  \href{https://doi.org/10.1007/s00208-008-0320-6}{10.1007/s00208-008-0320-6}.

\bibitem{felsenstein1981}
J.~Felsenstein.
\newblock Evolutionary trees from {DNA} sequences: A maximum likelihood
  approach.
\newblock {\em Journal of Molecular Evolution}, 17:368–376, 1981.
\newblock doi: \href{https://doi.org/10.1007/bf01734359}{10.1007/bf01734359}.

\bibitem{felsenstein}
J.~Felsenstein.
\newblock {\em Inferring Phylogenies}.
\newblock Oxford University Press Inc., New York, N.Y., 2nd edition, 2004.

\bibitem{sumnerPurine}
J.~Fernández-Sánchez, J.~Sumner, P.~Jarvis, and M.~Woodhams.
\newblock Lie {M}arkov models with purine/pyrimidine symmetry.
\newblock {\em Journal of Mathematical Biology}, 70(4):855–891, 2014.
\newblock doi:
  \href{https://doi.org/10.1007/s00285-014-0773-z}{10.1007/s00285-014-0773-z}.

\bibitem{grassman}
W.~Grassman.
\newblock Transient solutions in {M}arkovian queueing systems.
\newblock {\em Computers and Operations Research}, 4(1):47--53, 1977.
\newblock doi:
  \href{https://doi.org/10.1016/0305-0548(77)90007-7}{10.1016/0305-0548(77)90007-7}.

\bibitem{hky}
Masami Hasegawa, Hirohisa Kishino, and Taka aki Yano.
\newblock Dating of the human-ape splitting by a molecular clock of
  mitochondrial {DNA}.
\newblock {\em Journal of Molecular Evolution}, 22(2):160--174, 1985.

\bibitem{higham}
Nicholas~J. Higham.
\newblock {\em Functions of Matrices: Theory and Computation}.
\newblock Other Titles in Applied Mathematics. Society for Industrial and
  Applied Mathematics, 1 edition, 2008.

\bibitem{huelsenbeck}
J.~Huelsenbeck.
\newblock Bayesian phylogenetic model selection using reversible jump {M}arkov
  chain monte carlo.
\newblock {\em Molecular Biology and Evolution}, 21(6):1123–1133, 2004.
\newblock doi:
  \href{https://doi.org/10.1093/molbev/msh123}{10.1093/molbev/msh123}.

\bibitem{jukes}
T.~Jukes and C.~Cantor.
\newblock Evolution of protein molecules.
\newblock {\em Mammalian Protein Metabolism}, 3:21--132, 1969.
\newblock doi:
  \href{https://doi.org/10.1016/b978-1-4832-3211-9.50009-7}{10.1016/b978-1-4832-3211-9.50009-7}.

\bibitem{Kimura_1981}
M~Kimura.
\newblock Estimation of evolutionary distances between homologous nucleotide
  sequences.
\newblock {\em Proceedings of the National Academy of Sciences},
  78(1):454--458, 1981.

\bibitem{lanave}
Cecilia Lanave, Giuliano Preparata, Cecilia Sacone, and Gabriella Serio.
\newblock A new method for calculating evolutionary substitution rates.
\newblock {\em Journal of Molecular Evolution}, 20(1):86--93, 1984.

\bibitem{minh}
Bui~Quang Minh, Heiko~A Schmidt, Olga Chernomor, Dominik Schrempf, Michael~D
  Woodhams, Arndt von Haeseler, and Robert Lanfear.
\newblock {IQ}-{TREE} 2: New models and efficient methods for phylogenetic
  inference in the genomic era.
\newblock {\em Molecular Biology and Evolution}, 37(5):1530--1534, 2020.

\bibitem{mukhamedov}
F.~Mukhamedov.
\newblock Dobrushin ergodicity coefficient and weak ergodicity of {M}arkov
  chains on {J}ordan algebras.
\newblock {\em Journal of Physics: Conference Series}, 435(1), 2013.
\newblock doi:
  \href{https://doi.org/10.1088/1742-6596/435/1/012002}{10.1088/1742-6596/435/1/012002}.

\bibitem{pollett}
P.K. Pollett.
\newblock The generalized {K}olmogorov criterion.
\newblock {\em Stochastic Processes and their Applications}, 33(1):29--44,
  1989.

\bibitem{tavare}
Tavar\'e S. and Miura R.
\newblock Some probabilistic and statistical problems in the analysis of {DNA}
  sequences.
\newblock {\em Lectures on Mathematics in the Life Sciences}, 17:57--86, 1986.

\bibitem{sagan}
B.~Sagan.
\newblock {\em The symmetric group: representations, combinatorical algorithms,
  and symmetric functions}.
\newblock Springer, New York, N.Y., 2nd edition, 2001.
\newblock doi:
  \href{https://doi.org/10.1007/978-1-4757-6804-6}{10.1007/978-1-4757-6804-6}.

\bibitem{steel}
M.~Steel.
\newblock {\em Phylogeny: Discrete and Random Processes in Evolution}.
\newblock Society for Industrial and Applied Mathematics, New York, N.Y., 2016.
\newblock doi:
  \href{https://doi.org/10.1137/1.9781611974485}{10.1137/1.9781611974485}.

\bibitem{sumnerClosed}
J.~Sumner.
\newblock Multiplicatively closed {M}arkov models must form {L}ie algebras.
\newblock {\em The ANZIAM Journal}, 59(2):240--246, 2017.
\newblock doi:
  \href{https://doi.org/10.21914/anziamj.v59i0.12028}{10.21914/anziamj.v59i0.12028}.

\bibitem{sumnerLie}
J.~Sumner, J.~Fernández-Sánchez, and P.~Jarvis.
\newblock Lie {M}arkov models.
\newblock {\em Journal of Theoretical Biology}, 298:16--31, 2012.
\newblock doi:
  \href{https://doi.org/10.1016/j.jtbi.2011.12.017}{10.1016/j.jtbi.2011.12.017}.

\bibitem{woodhams}
J.~Sumner and M.~Woodhams.
\newblock Lie-{M}arkov models derived from finite semigroups.
\newblock {\em Bulletin of Mathematical Biology}, 81(2):361--383, 2019.
\newblock doi:
  \href{https://doi.org/10.1007/s11538-018-0455-x}{10.1007/s11538-018-0455-x}.

\bibitem{dnaRecomb}
S.~Sverchkov.
\newblock The structure and representation of n-ary algebras of {DNA}
  recombination.
\newblock {\em Central European Journal of Mathematics}, 9, 2011.
\newblock doi:
  \href{https://doi.org/10.2478/s11533-011-0087-y}{10.2478/s11533-011-0087-y}.

\bibitem{tamuraNei}
K.~Tamura and M.~Nei.
\newblock Estimation of the number of nucleotide substitutions in the control
  region of mitochondrial {DNA} in humans and chimpanzees.
\newblock {\em Molecular Biology and Evolution}, 10(3):512--526, 1993.
\newblock doi:
  \href{https://doi.org/10.1093/oxfordjournals.molbev.a040023}{10.1093/oxfordjournals.molbev.a040023}.

\bibitem{genAlgebra}
A.~Worz-Busekros.
\newblock {\em Lecture Notes in Mathematics: Algebras in Genetics}.
\newblock Springer-Verlag Berlin Heidelberg, 1st edition, 1980.
\newblock doi:
  \href{https://doi.org/10.1007/978-3-642-51038-0}{10.1007/978-3-642-51038-0}.

\end{thebibliography}

\end{document}